\documentclass[twoside]{article}
\usepackage{thmtools}
\usepackage{amsthm}
\usepackage{amsmath,amssymb,mathtools}
\usepackage{mathrsfs}
\usepackage{enumerate}
\usepackage{dsfont}
\usepackage{xcolor}
\usepackage{extarrows}
\usepackage{bm}
\usepackage{multirow}
\usepackage{float}

\usepackage{PRIMEarxiv}

\usepackage[utf8]{inputenc} 
\usepackage[T1]{fontenc}    
\usepackage{hyperref}       
\usepackage{url}            
\usepackage{booktabs}       
\usepackage{amsfonts}       
\usepackage{nicefrac}       
\usepackage{microtype}      
\usepackage{lipsum}
\usepackage{subcaption}
\usepackage{fancyhdr}       
\usepackage{graphicx}       
\graphicspath{{media/}}     

\newtheoremstyle{break}
  {10pt}{10pt}
  {\normalfont}
  {}
  {\bfseries}
  {.}
  {\newline}
  {}
\theoremstyle{plain}
\newtheorem{theorem}{Theorem}[section]
\newtheorem{lemma}[theorem]{Lemma}

\newtheorem{definition}[theorem]{Definition}

\newtheorem{remark}[theorem]{Remark}

\newcounter{assumpA}
\renewcommand{\theassumpA}{A\arabic{assumpA}}

\newenvironment{assumptionA}[1][]{%
  \refstepcounter{assumpA}%
  \par\noindent
  \textbf{(\theassumpA)%
  \if\relax\detokenize{#1}\relax\else\quad #1\fi.}%
  \quad
  \itshape
}{\par}

\makeatletter
\renewenvironment{proof}[1][\proofname]{%
  \par
  \pushQED{\qed}%
  \normalfont
  \topsep 1pt plus 1pt minus 1pt
  \trivlist
  \item[\hskip\labelsep\itshape #1\@addpunct{.}]%
  \ignorespaces
}{%
  \popQED\endtrivlist\@endpefalse
}
\makeatother

\newtheoremstyle{proofappendix}
  {6pt}
  {6pt}
  {\normalfont}
  {}
  {\bfseries}
  {~~of}
  {0.5em}
  {}

\theoremstyle{proofappendix}

\numberwithin{equation}{section}

\title{Elasticity estimation via the CKLS--CIR transform for small-noise CEV-type diﬀusions}

\author{
  Boyuan Ning \\
  Graduate School of Fundamental Science and Engineering \\
  Waseda University \\
  3-4-1 Ohkubo, Shinjuku-ku, Tokyo, Japan\\
  \texttt{ningboyuan@akane.waseda.jp}\\
  \And
  Yasutaka Shimizu \\
  Department of Applied Mathematics \\
  Waseda University \\
  3-4-1 Ohkubo, Shinjuku-ku, Tokyo, Japan\\
  \texttt{shimizu@waseda.jp}\\
}

\begin{document}
\maketitle

\begin{abstract}
We construct a feasible two-step estimator for the elasticity parameter in small-noise CEV-type diffusions. Under joint small-noise and high-frequency asymptotics, we establish its consistency and an $\varepsilon^{-1}$ central limit theorem with an explicit asymptotic variance. We adapt a Lamperti-based transformation strategy to elasticity inference for a CEV process with an attainable, absorbing $0$ boundary. Under a local-to-CEV scaling, an auxiliary CKLS process and a Girsanov change of measure yield a strictly positive CIR-type benchmark in which the elasticity is recast as a linear-drift parameter. A preliminary elasticity estimator based on local realized variance supplies the unknown parameter in the state mapping, yielding a feasible LSE-type estimator. The asymptotic theory combines control of the CKLS-to-CEV path-replacement error, total-variation closeness between the original and changed measures, and asymptotic negligibility of the preliminary plug-in error, thereby transferring the auxiliary CIR limit theory to the feasible estimator under the original measure.

\end{abstract}


\section{Introduction}
\subsection{CEV and CKLS-type diffusions}
Throughout this paper, let $(\Omega,\mathcal{F},\{\mathcal{F}_{t}\}_{t\geq 0},\mathbb{P})$ be a filtered probability space satisfying the usual conditions, and let $W=\{W_{t}\}_{t\geq0}$ be a standard Brownian motion defined on this space. 

\begin{definition}[CEV model]\label{def:CEV}
Let $\mu\in\mathbb{R}$, $\nu>0$ and $\vartheta>0$. A constant elasticity of variance process, first developed by \cite{cox1975notes} and abbreviated as a CEV process, is a process 
$X=\{X_{t}\}_{t\geq 0}$ satisfying
\[
	dX_{t}=\mu X_{t}dt+\nu\big(X_{t}\big)^{\vartheta}dW_{t},\qquad X_{0}=x_{0}>0.
\]
Here $\mu$ is the linear-drift coefficient, $\nu$ is the volatility scale parameter, and $\vartheta$ is the elasticity parameter. For the standard analytical properties of the CEV process, including existence, positivity, boundary behavior, moment properties, and ergodicity, see, e.g., \cite{revuz1999continuous}, \cite{delbaen2002note}, \cite{gontis2011interplay}, \cite{lindsay2012simulation}, \cite{tsumurai2020malliavin}, \cite{ning2025ckls}.

In this paper, we focus on the sublinear-elasticity, mean-decaying regime
\[
	\frac{1}{2}<\vartheta<1,\qquad\mu<0.
\]
where $\vartheta$ is the parameter of interest, while $\mu$ is treated as a known linear-drift parameter. Notably, for this regime of $\vartheta$, the CEV equation admits a pathwise unique global strong solution satisfying
$\mathbb{P}(0\leq X_{t}<\infty\text{ for all }t\geq0)=1$. Starting from $x_{0}>0$, the process can hit $0$ with positive probability, and once it reaches $0$, it remains there thereafter.
\end{definition}
The restriction on $\vartheta$ is motivated by both theoretical and empirical considerations. First, this range is the CEV analogue of the CKLS elasticity regime studied in \cite{ning2025ckls}. For the CEV model, this range combines strong well-posedness with an attainable, absorbing $0$ boundary, which is excluded by the boundary assumptions in \cite{ait2002maximum}. Drawing on the Lamperti-based perspective of that work, we developed an auxiliary CKLS process and a change of measure based on the main result of \cite{ning2025ckls} to conduct asymptotic inference through a strictly positive CIR benchmark. Second, this regime is empirically relevant. Among econometric studies, we mention three illustrative examples. \cite{xiao2017estimating} report elasticity estimates of 0.880 and 0.885 for the CSI 300 index using MCMC and maximum likelihood, while \cite{ballestra2016numerical} find that 27 of the 50 NYSE equities in their sample have estimated elasticities within this interval. Similarly, \cite{ballestra2011constant} obtain maximum-likelihood estimates within this interval for 16 of the 40 Italian equities in their sample.

The restriction on $\mu$ corresponds to a mean-decaying deterministic linear-drift in the CEV model and, under the local-to-CEV formulation \eqref{Local-to-CEV-CKLS}, places the associated CKLS model in the mean-reverting regime; see Definition~\ref{def:CKLS}. This condition is consistent with the mean-reverting CKLS benchmark underlying the CKLS--CIR transform given in Appendix~\ref{aaa}, in the sense that it ensures regularity properties of the CKLS process required in the corresponding analysis and guarantees that the transformed pseudo-CIR dynamics remain mean-reverting. For $\vartheta\in(1/2,1)$, after the CKLS--CIR transform, the corresponding CIR-type linear-drift parameter (see Definition~\ref{def:CIR}) is $\beta(\vartheta)=-2\mu(1-\vartheta)>0$.

For statistical inference, we consider the following target model: the small-noise CEV model
\begin{equation}\label{small-noise-CEV-def}
    dX_{t}^{\varepsilon,\vartheta}=\mu X_{t}^{\varepsilon,\vartheta}dt+\varepsilon\nu (X_{t}^{\varepsilon,\vartheta})^{\vartheta}dW_{t},\qquad X_{0}^{\varepsilon,\vartheta}=x_{0}>0,
\end{equation}
where $\varepsilon\in(0,1]$ is the noise-level parameter, $\mu<0$, and $\nu>0$. Our objective is to estimate the elasticity parameter $\vartheta$ from high-frequency observations over a fixed time horizon under joint small-noise and high-frequency asymptotics.

\begin{definition}[CKLS model]\label{def:CKLS}
Let $a,b\in\mathbb{R}$, $\sigma>0$ and $k>0$. 
A Chan--Karolyi--Longstaff--Sanders process, first developed by \cite{chan1992empirical} and abbreviated as a CKLS process, is a process $Y=\{Y_{t}\}_{t\geq 0}$ satisfying
\[
	dY_{t}=(a-bY_{t})dt+\sigma(Y_{t})^{k}dW_{t},\qquad Y_{0}=y_{0}>0.
\]
Here $a$ is the constant-drift coefficient and $b$ is the linear-drift coefficient. In the mean-reverting regime $a,b>0$, $b$ is the mean-reversion speed and $a/b$ is the long-run mean. $\sigma$ is the volatility scale parameter, and $k$ is the elasticity parameter. For $a>0$ and $k\in(1/2,1)$, the equation admits a pathwise unique global strong solution satisfying $\mathbb{P}(0<Y_{t}<\infty\text{ for all }t\geq0)=1$. For other analytical properties of the CKLS process, including existence, positivity, boundary behavior, moment properties, and ergodicity, see, e.g., \cite{andersen2007moment}, \cite{tsumurai2020malliavin}, \cite{mishura2022parameter}, \cite{ning2025ckls}.

While the CEV model appears to be a special case of the CKLS model with $a=0$, it has distinct analytical properties. Although the diffusion coefficient degenerates at the origin in both models, for $1/2<\vartheta<1$, the origin is attainable and absorbing for the CEV process, whereas the CKLS process remains strictly positive.
\end{definition}

Throughout this paper, the associated CKLS model serves as a benchmark and is considered in the mean-reverting power-volatility regime $a>0$, $b>0$ (ergodic case), $\sigma>0$, and $k\in(1/2,1)$, and serves primarily as a local benchmark and technical device. For the CKLS–CIR transform and related background, see \cite{ning2025ckls}. Accordingly, we adopt the following specification of the CIR model for the mean-reverting case.

\begin{definition}[CIR model]\label{def:CIR}
Let $\alpha>0$, $\beta>0$ and $\gamma>0$. A Cox--Ingersoll--Ross process, first developed by \cite{cox1985theory} and abbreviated as a CIR process, is a process $Z=\{Z_{t}\}_{t\geq 0}$ satisfying
\[
	dZ_{t}=(\alpha-\beta Z_{t})dt+\gamma\sqrt{Z_{t}}dW_{t},\qquad Z_{0}=z_{0}>0.
\]
Here $\beta$ denotes the mean-reversion speed, $\alpha$ represents the long-run mean multiplied by $\beta$, and $\gamma$ is the volatility scale parameter. For the standard analytical properties of the CIR process, including existence, positivity, boundary behavior, moment properties, and ergodicity, see, e.g., \cite{revuz1999continuous}, \cite{jeanblanc2009mathematical}, \cite{alaya2013asymptotic}, \cite{ning2025ckls}.

Notably, the condition $\alpha>0$ gives an inward drift at $0$ and prevents absorption there; together with $\beta>0$, it yields a unique Gamma invariant distribution and ergodicity. Strict positivity of paths starting from $z_{0}>0$ is ensured by the additional Feller condition $2\alpha\geq\gamma^{2}$.

We also highlight the CEV--CIR correspondence: Let $X$ be the CEV process in Definition~\ref{def:CEV} with $\vartheta\neq1$.
Set $Z_{t}:=X_{t}^{2(1-\vartheta)}$ and $\tau_{0}:=\inf\{t\geq0:X_{t}=0\}$. Then, for $t<\tau_{0}$, following from It\^{o}'s formula applied to $x\mapsto x^{2(1-\vartheta)}$ localized on compact subintervals of $(0,\infty)$, $Z$ satisfies the following CIR-type dynamics, whose drift parameters need not satisfy the positivity conditions in Definition~\ref{def:CIR}:
\begin{equation}\label{prop:CEV-CIR}
    dZ_{t}=(\alpha_{\vartheta}-\beta_{\vartheta}Z_{t})dt+\gamma_{\vartheta}\sqrt{Z_{t}}dB_{t},\text{ where }\alpha_{\vartheta}=\nu^{2}(1-\vartheta)(1-2\vartheta),~\beta_{\vartheta}=-2\mu(1-\vartheta),~\gamma_{\vartheta}=2\nu|1-\vartheta|,
\end{equation}
$B_{t}=\operatorname{sgn}(1-\vartheta)W_{t}$ is a standard Brownian motion. This power transform is a scaled square of the Lamperti coordinate
\[
    \ell_{\vartheta}(x)=\frac{x^{1-\vartheta}}{\nu(1-\vartheta)},\qquad x^{2(1-\vartheta)}=\nu^{2}(1-\vartheta)^{2}\ell_{\vartheta}(x)^{2}.
\]
Consequently, the CEV process admits the power representation
$X_{t}=Z_{t}^{1/[2(1-\vartheta)]}$ before $\tau_{0}$. This CIR-type representation does not guarantee strict positivity for all times: when $1/2<\vartheta<1$, $\alpha_{\vartheta}<0$ and $2\alpha_{\vartheta}-\gamma_{\vartheta}^{2}=-2\nu^{2}(1-\vartheta)<0$, so the CIR-type diffusion admits a negative constant-drift and the Feller condition is not satisfied.
\end{definition}
Despite their different drift specifications, the CEV, CKLS, and CIR models share the power-type diffusion coefficient $x\mapsto cx^{\ell}$, with $(c,\ell)=(\nu,\vartheta)$, $(\sigma,k)$, and $(\gamma,1/2)$, respectively. For elasticity inference, the CEV and CKLS models hence share the same essential target—the power index in the diffusion coefficient—which motivates the term “CKLS-type diffusions” and our local-to-CEV CKLS formulation \eqref{Local-to-CEV-CKLS}.

\subsection{Literature review} 

\subsubsection{Elasticity Inference for the CEV Model}
Parametric inference for discretely observed diffusion processes has been studied extensively, particularly under high-frequency sampling. Major approaches include Gaussian contrast, quasi-likelihood, and estimating-function methods, with local Gaussian approximations playing an important role in high-frequency inference. Representative contributions include \cite{florens1989approximate}, \cite{yoshida1992estimation}, \cite{genon1993estimation}, \cite{genon1994estimation}, \cite{bibby1995martingale}, \cite{kessler1999estimating}, \cite{jacobsen2001discretely}, \cite{jacobsen2002optimality}, \cite{uchida2013quasi}, and \cite{jakobsen2017efficient}.

In the regime $1/2<\vartheta<1$, the CEV model presents additional difficulties for high-frequency inference near $0$: $\frac{d}{dx}x^{\vartheta}=\vartheta x^{\vartheta-1}\to\infty$ as $x\downarrow0$, so the diffusion coefficient is not Lipschitz at the origin, complicating local expansions, Euler-type quasi-likelihood arguments, and uniform remainder estimates near this boundary. Moreover, unlike in the mean-reverting CKLS model with a positive constant-drift, both drift and diffusion coefficients of the CEV model vanish at $0$, making the origin absorbing; its boundary behavior therefore requires explicit treatment (see, e.g., \cite{delbaen2002note,lindsay2012simulation}).

One possible approach is to exploit analytical representations of the CEV transition law through squared-Bessel-type transforms and modified Bessel functions; see, e.g., \cite{delbaen2002note}. However, the complexity of the resulting density makes explicit elasticity estimation and high-frequency asymptotic expansions technically demanding. Approximate likelihood methods based on local linearization
\cite{shoji1998estimation} or Lamperti transforms and Hermite-type expansions \cite{ait2002maximum} improve tractability under suitable regularity conditions. For CEV elasticity inference, the Lamperti transform itself depends on the unknown elasticity, adding parameter-derivative issues to the likelihood approximation. Moreover, absorption at $0$ requires boundary analysis beyond the non-attainable-boundary framework of \cite{ait2002maximum}.

Realized-volatility, power-variation, and bipower-variation methods estimate pathwise volatility functionals from high-frequency data \cite{barndorff2002econometric,barndorff2004power}. In particular, realized variance uses squared increments to estimate integrated variance.

Existing approaches to CEV elasticity estimation include model-specific regression, numerical likelihood, and option-implied calibration. \cite{beckers1980constant} proposed log-linear regression based on the variance--level relationship. Subsequent studies numerically optimized return-based, exact transition-density, or Gaussian approximate likelihoods \cite{tucker1988tests,ballestra2011constant,beiranvand2021new}, while others calibrated elasticity by minimizing option-pricing or implied-volatility errors \cite{emanuel1982further,tucker1988tests,yuen2001estimation,ballestra2016numerical}. These studies primarily focus on numerical estimation and empirical calibration using financial data, rather than on high-frequency asymptotic theory for elasticity estimation.

\subsubsection{Parametric Inference for Small-Noise Diffusions}
Statistical inference for dynamical systems with small noise was systematically developed by \cite{kutoyants1994identification}, covering parametric and nonparametric identification under fixed-horizon vanishing-noise asymptotics. For incomplete or discrete observations, \cite{genon1990maximnm} established asymptotically normal maximum-contrast estimators, while \cite{laredo1990sufficient} studied asymptotic sufficiency in Le Cam's sense and constructed estimators equivalent to those based on continuous observation.

Martingale estimating functions for discretely observed diffusions \cite{bibby1995martingale} were extended to small-dispersion asymptotics using stochastic Taylor expansions and Gaussian approximations \cite{sorensen2000small}. Under joint small-noise and high-frequency asymptotics, \cite{sorensen2003small} constructed minimum-contrast estimators for drift and diffusion parameters in multidimensional models based on local Gaussian transition approximations. Related contributions include explicit contrasts for joint drift--diffusion estimation encompassing quadratic martingale estimating functions \cite{gloter2009estimation}. Using a Gaussian process approximation, \cite{guy2014parametric} developed minimum-contrast estimators for discretely observed small-noise diffusions and established their consistency and asymptotic normality. More recently, \cite{kawai2023adaptive} developed adaptive estimators and tests for drift and diffusion parameters under joint small-noise and high-frequency asymptotics.

For drift estimation under additive small L\'{e}vy noise, \cite{long2013least} proposed a simple least-squares contrast. The least-squares approach was extended to state-dependent small L\'{e}vy noise without imposing a Lipschitz condition
on the dispersion coefficient \cite{long2017least}. For jump-noise models, \cite{shimizu2017threshold} further developed filtering and threshold procedures.

\subsection{Contributions and paper organization} 
The main contribution of this paper is the development and asymptotic analysis of a two-step high-frequency estimator for the CEV elasticity parameter, obtained by embedding the small-noise CEV structure into a small-noise CKLS framework under a local-to-CEV scaling. First, the local variance relation induced by the CEV diffusion coefficient yields a preliminary elasticity estimator based on high-frequency quadratic-variation information. Second, a $C^{2}$ state mapping, together with a Girsanov change of measure, recasts the elasticity parameter of the auxiliary CKLS process as a linear-drift parameter in a small-noise CIR-type model. Under the local-to-CEV scaling, the auxiliary CKLS and target CEV paths are asymptotically close, allowing the observed CEV path to replace the auxiliary path
in an LSE-type linear-drift estimator motivated by
\cite{long2013least,long2017least}. Substituting the preliminary estimate into the state mapping then makes the procedure feasible. Under our joint asymptotic conditions, we control the path-replacement and plug-in errors separately and show that the $o_{\mathbb{P}}(\varepsilon)$ accuracy of the preliminary estimator preserves the oracle limit distribution. We thereby establish consistency and an $\varepsilon^{-1}$ central limit theorem for the feasible estimator under the original CEV observation model, with an explicit asymptotic variance, providing a distributional foundation for transformation-based inference on CEV elasticity.

As a further methodological contribution, we draw on the transform-based perspective exemplified by \cite{ait2002maximum} to develop a measure-change-based route to inference for a CEV process that may hit $0$ and is absorbed thereafter. A\"{i}t-Sahalia uses a Lamperti transform to facilitate likelihood expansions under boundary conditions ensuring non-attainability. Here, a rescaled square of the Lamperti transform is combined with an auxiliary CKLS process and a change of measure, allowing localized asymptotic analysis through strictly positive CIR dynamics. The unknown elasticity in the state mapping is replaced by a preliminary estimator, and the resulting plug-in error is proved negligible at the first-order scale under joint small-noise and high-frequency asymptotics. Under the local-to-CEV scaling, the Girsanov density converges to $1$ in $L^{1}$, making the original and changed measures asymptotically close in total variation. Together with the auxiliary-to-CEV approximation and plug-in equivalence, this permits the transfer of consistency and an explicit Gaussian limit to the feasible estimator under the original CEV observation model with absorption at $0$.

The remainder of the paper is organized as follows. Section~\ref{sec2} introduces the assumptions and constructs the two-step estimator. Section~\ref{sec3} states and proves the main asymptotic result and discusses the zero-boundary issue under direct CEV mapping. Appendix~\ref{aaa} develops the CKLS--CIR transformation and change of measure, and establishes asymptotic measure transfer under the local-to-CEV scaling. Appendix~\ref{bbb} develops the technical analysis supporting the two-step estimator and the main asymptotic results.

\section{Assumptions and Estimator Construction}\label{sec2}
\subsection{Assumptions}\label{subsec:ass}

\begin{assumptionA}[High-frequency small-noise scheme]\label{ass:A1}
Take the target small-noise CEV process $X^{\varepsilon,\vartheta_{0}}$ to be the nonnegative CEV solution absorbed at $0$. It is observed at the equidistant grid points, with a finite observation horizon $T>0$,
\[
    t_{i}:=i\Delta_{n},\qquad i=0,1,\ldots,n,\qquad\Delta_{n}:=T/n.
\]
The joint asymptotic regime is an increasing sampling frequency and a vanishing noise
\[
    \varepsilon\to0,\qquad n\to\infty \quad (\text{equivalently, }\Delta_{n}\to0).
\]
\end{assumptionA}

\begin{assumptionA}[CEV parameter space]\label{ass:A2}
Of interest is the elasticity parameter $\vartheta_{0}$ of the small-noise CEV model, satisfying
\[
    1/2<\underline{\vartheta}<\vartheta_{0}<\overline{\vartheta}<1,\qquad\Theta:=[\underline{\vartheta},\overline{\vartheta}].
\]
Moreover, $\mu<0$, $\nu>0$, and $x_{0}>0$ are assumed to be known and fixed constants.
\end{assumptionA}

\begin{assumptionA}[Auxiliary regularization]\label{ass:A3}
A positive regularization constant $a^{*}$ is fixed in advance, and we set
\[
    a_{\varepsilon}:=\varepsilon^{2}a^{*}>0,\qquad\text{for each }\varepsilon>0.
\]
This regularization is used in the auxiliary CKLS process introduced in \eqref{Local-to-CEV-CKLS}.
\end{assumptionA}

\begin{assumptionA}[Local window and identification]\label{ass:A4}
Fix a reference time $\tau\in[0,T)$ and define
\[
    j_{n}:=\lfloor\tau/\Delta_{n}\rfloor,\qquad\tau_{n}:=t_{j_{n}},\qquad h_{n}:=m_{n}\Delta_{n},
\]
where $m_{n}$ is a sequence of positive integers such that
$\tau'_{n}:=\tau_{n}+h_{n}\leq T$. Then $\tau_{n}\to\tau$. The local window $h_{n}$ satisfies
\[
    h_{n}\to0,\qquad m_{n}\to\infty,\qquad (\text{or equivalently,}\quad\Delta_{n}/h_{n}=1/m_{n}\to0).
\]
Moreover, choose $\tau$ such that $|\log\mathcal{X}_{\tau}|\geq c_{0}$ for some constant $c_{0}>0$, where $\mathcal{X}_{t}:=x_{0}e^{\mu t},~0\leq t\leq T$, denotes the deterministic small-noise limit of the target small-noise CEV model.
\end{assumptionA}

\begin{assumptionA}[Feasibility rate condition]\label{ass:A5}
The local window $h_{n}$, the observation mesh $\Delta_{n}$, and the small-noise level $\varepsilon$ satisfy
\[
    \sqrt{\Delta_{n}/h_{n}}=o(\varepsilon),\qquad\Delta_{n}=o(\varepsilon^{3})\quad (\text{equivalently, }n\varepsilon^{3}\to\infty~\text{, since }\Delta_{n}=T/n\text{ with fixed }T)
\]
\end{assumptionA}

\subsection{Preliminary elasticity estimation}\label{subsec:preliminary-estimation}
Under Assumption~\ref{ass:A4}, for the target CEV process $X^{\varepsilon,\vartheta_{0}}$, define the local realized variation by
\[
    \widehat{\mathcal{V}}_{\varepsilon,n}:=\sum_{i=j_{n}+1}^{j_{n}+m_{n}}\Big(X_{t_{i}}^{\varepsilon,\vartheta_{0}}-X_{t_{i-1}}^{\varepsilon,\vartheta_{0}}\Big)^{2}.
\]
To account for the deterministic within-window evolution $u\mapsto e^{\mu u}X_{\tau_{n}}^{\varepsilon,\vartheta_{0}}$, for $\vartheta\in\Theta$ define
\[
    \Phi_{n}(\vartheta):=\frac{1}{h_{n}}\int_{0}^{h_{n}}e^{2\mu\vartheta u}du=\frac{e^{2\mu\vartheta h_{n}}-1}{2\mu\vartheta h_{n}}.
\]
By Lemma~\ref{lem:CEV-holder-continuity}, over a sufficiently short time interval $[\tau_{n},\tau'_{n}]\subseteq[0,T]$, the local quadratic variation satisfies
\[
    \frac{[X^{\varepsilon,\vartheta_{0}}]_{\tau'_{n}}-[X^{\varepsilon,\vartheta_{0}}]_{\tau_{n}}}{\varepsilon^{2}\nu^{2}h_{n}}=\frac{1}{h_{n}}\int_{\tau_{n}}^{\tau'_{n}}\big(X_{s}^{\varepsilon,\vartheta_{0}}\big)^{2\vartheta_{0}}ds=\big(X_{\tau_{n}}^{\varepsilon,\vartheta_{0}}\big)^{2\vartheta_{0}}\Phi_{n}(\vartheta_{0})+O_{\mathbb{P}}(\varepsilon\sqrt{h_{n}}).
\]
Replacing the quadratic-variation increment by $\widehat{\mathcal{V}}_{\varepsilon,n}$ and taking logarithms motivates the preliminary estimator below. To ensure that the logarithmic criterion is well defined and to maintain nondegenerate local identification, define $\mathcal{E}_{\varepsilon,n}:=\big\{\widehat{\mathcal{V}}_{\varepsilon,n}>0,~X_{\tau_{n}}^{\varepsilon,\vartheta_{0}}\in(0,e^{-c_{0}/2}]\cup[e^{c_{0}/2},\infty)\big\}$. On $\mathcal{E}_{\varepsilon,n}$, define the raw preliminary
estimator by
\[
    \widetilde{\vartheta}_{\varepsilon,n}^{*}:=\operatorname*{arg~min}_{\vartheta\in\Theta}\Big[\log\frac{\widehat{\mathcal{V}}_{\varepsilon,n}}{\varepsilon^{2}\nu^{2}h_{n}}-2\vartheta\log X_{\tau_{n}}^{\varepsilon,\vartheta_{0}}-\log \Phi_{n}(\vartheta)\Big]^{2}.
\]
If the minimizer is not unique, select the smallest one. Fix an arbitrary $\vartheta^{\diamond}\in\Theta$. The preliminary estimator is defined by
\begin{equation}\label{preliminary-estimator}
    \widetilde{\vartheta}_{\varepsilon,n}:=
    \begin{cases}
        \widetilde{\vartheta}_{\varepsilon,n}^{*},
        & \text{on }\mathcal{E}_{\varepsilon,n},\\[1ex]
        \vartheta^{\diamond},
        & \text{on }\mathcal{E}_{\varepsilon,n}^{c}.
    \end{cases}
\end{equation}
\subsection{Local-to-CEV reduction and state mapping}\label{subsec:local-to-cev}
To motivate estimation through the transformed linear drift, we introduce a strictly positive auxiliary CKLS process as a theoretical benchmark. Under Assumption~\ref{ass:A3}, consider
\begin{equation}\label{Local-to-CEV-CKLS}
	dY_{t}^{\varepsilon,\vartheta}=\big(\varepsilon^{2}a^{*}+\mu Y_{t}^{\varepsilon,\vartheta}\big)dt+\varepsilon\nu\big(Y_{t}^{\varepsilon,\vartheta}\big)^{\vartheta}dW_{t},\qquad Y_{0}^{\varepsilon,\vartheta}=x_{0}>0.
\end{equation}
The additional vanishing local-to-CEV regularization $\varepsilon^{2}a^{*}$ ensures the CKLS strict positivity required for the CKLS--CIR transform (see Remark~\ref{rem:ckls-parameter-restrictions}) for each fixed $\varepsilon>0$ while keeping its first-order small-noise behavior the same as the target CEV process: at $\vartheta=\vartheta_{0}$, couple the two processes using the same Brownian motion and initial value. Lemma~\ref{lem:uniform-deterministic-limit} shows that both processes remain in a common compact subset of $(0,\infty)$ over $[0,T]$ with probability tending to $1$. Combining this localization with the $L^{2}$ approximation bound in Lemma~\ref{lem:local-to-CEV-process-approx} yields
\begin{equation}\label{process-level-estiamte}
	\sup_{0\leq t\leq T}\big|Y_{t}^{\varepsilon,\vartheta_{0}}-X_{t}^{\varepsilon,\vartheta_{0}}\big|=O_{\mathbb{P}}(\varepsilon^{2})=o_{\mathbb{P}}(\varepsilon).
\end{equation}
For $x\geq 0$ and $\vartheta\in\Theta$, define
\[
    \mathcal{T}(x;\vartheta):=\frac{\zeta^{2}}{4(1-\vartheta)^{2}}x^{2(1-\vartheta)},\qquad\zeta>0,
\]
This mapping is continuous on $[0,\infty)$ and $C^{2}$ on
$(0,\infty)$, with $\mathcal{T}(0;\vartheta)=0$. Here, $\zeta$ is an arbitrary constant that determines the scaling of the transformed state variable and cancels from the final elasticity estimator and its asymptotic distribution. Moreover, this mapping $\mathcal{T}$ is a rescaled square of the Lamperti transform and is chosen in accordance with the CKLS--CIR construction developed in Appendix~\ref{aaa}. 

At the true elasticity value $\vartheta_{0}$, define the transformed processes
\[
    Z_{t}^{Y,\varepsilon,\circ}:=\mathcal{T}\big(Y_{t}^{\varepsilon,\vartheta_{0}};\vartheta_{0}\big),\qquad Z_{t}^{X,\varepsilon,\circ}:=\mathcal{T}\big(X_{t}^{\varepsilon,\vartheta_{0}};\vartheta_{0}\big).
\]
The CKLS--CIR transform results developed in Appendix~\ref{aaa}, comprising the state mapping and the Girsanov change of measure, show that under an auxiliary equivalent measure $\mathbb{Q}^{\varepsilon}$, with $\widetilde{W}^{\varepsilon,\vartheta_{0}}$ a standard Brownian motion, $Z^{Y,\varepsilon,\circ}$ satisfies the following strictly positive CIR-type dynamics:
\[
    dZ_{t}^{Y,\varepsilon,\circ}=\big(\alpha_{\varepsilon}-\beta_{0}Z_{t}^{Y,\varepsilon,\circ}\big)dt+\varepsilon\gamma\sqrt{Z_{t}^{Y,\varepsilon,\circ}}d\widetilde{W}_{t}^{\varepsilon},\quad\text{where}\quad\beta_{0}=-2\mu(1-\vartheta_{0}), \quad\gamma=\nu\zeta, \quad\alpha_{\varepsilon}=\varepsilon^{2}\nu^{2}\zeta^{2}/2.
\]
Note that this is the boundary case $2\alpha_{\varepsilon}=\varepsilon^{2}\nu^{2}\zeta^{2}=\varepsilon^{2}\gamma^{2}$ of Feller condition. The auxiliary process and the Girsanov change of measure together provide a strictly positive CIR benchmark for localized asymptotic analysis. This allows the establishment of the oracle asymptotics without directly confronting the boundary degeneracy associated with the target CEV process hitting and being absorbed at $0$.

The resulting drift--elasticity relation, providing an indirect route to estimating $\vartheta$ through $\beta(\vartheta)$, is
\begin{equation}\label{elasticity-beta-relation}
    \vartheta=1+\frac{\beta(\vartheta)}{2\mu},\qquad\beta(\vartheta)=-2\mu(1-\vartheta)>0.
\end{equation}

\subsection{Feasible estimator and oracle benchmarks}\label{subsec:feasible-oracles}
Applying $\mathcal{T}$ to the observed CEV data with the preliminary estimate gives the feasible transformed observations
\begin{equation}\label{feasible-transformed-data}
    \widehat{Z}_{t_{i}}^{\varepsilon}:=\mathcal{T}\big(X_{t_{i}}^{\varepsilon,\vartheta_{0}};\widetilde{\vartheta}_{\varepsilon,n}\big),\qquad i=0,\ldots,n.
\end{equation}
The feasible LSE-type estimator of the transformed linear-drift parameter is
\[
    \widehat{\beta}_{\varepsilon,n}:=\frac{\alpha_{\varepsilon}\Delta_{n}\sum_{i=1}^{n}\widehat{Z}_{t_{i-1}}^{\varepsilon}-\sum_{i=1}^{n}\widehat{Z}_{t_{i-1}}^{\varepsilon}\Delta_{i}\widehat{Z}^{\varepsilon}}{\Delta_{n}\sum_{i=1}^{n}\big(\widehat{Z}_{t_{i-1}}^{\varepsilon}\big)^{2}},\quad\text{with}\quad\alpha_{\varepsilon}:=\frac{\varepsilon^{2}\nu^{2}\zeta^{2}}{2}\quad\text{and}\quad\Delta_{i}\widehat{Z}^{\varepsilon}:=\widehat{Z}_{t_{i}}^{\varepsilon}-\widehat{Z}_{t_{i-1}}^{\varepsilon}.
\]
Finally, by \eqref{elasticity-beta-relation}, the proposed two-step estimator of the CEV elasticity is
\begin{equation}\label{final-elasticity-estimator}
    \widehat{\vartheta}_{\varepsilon,n}:=1+\frac{\widehat{\beta}_{\varepsilon,n}}{2\mu}.
\end{equation}
Following \cite{long2013least,long2017least} (see also, \cite{shimizu2017threshold}), we adopt the unweighted Euler-type least-squares construction for discretely observed small-noise processes. This construction is convenient here because the observed CEV process may be absorbed at $0$. It avoids division by individual observations, while its denominator remains strictly positive because it includes the transformed initial observation $\widehat{Z}_{0}^{\varepsilon}=\mathcal{T}(x_{0},\widetilde{\vartheta}_{\varepsilon,n})=[\zeta^{2}x_{0}^{2(1-\widetilde{\vartheta}_{\varepsilon,n})}]/[4(1-\widetilde{\vartheta}_{\varepsilon,n})^{2}]>0$. 

For a discrete path $U=\{U_{t_{i}}\}_{i=0}^{n}$, write $\Delta_{i}U:=U_{t_{i}}-U_{t_{i-1}}$ and define the path functional
\begin{equation}\label{path-functional-B} 
    \mathcal{B}_{\varepsilon,n}(U):=\frac{\alpha_{\varepsilon}\Delta_{n}\sum_{i=1}^{n}U_{t_{i-1}}-\sum_{i=1}^{n}U_{t_{i-1}}\Delta_{i}U}{\Delta_{n}\sum_{i=1}^{n}\big(U_{t_{i-1}}\big)^{2}}.
\end{equation}
The CEV-based oracle linear-drift estimator, the auxiliary oracle linear-drift estimator and the feasible LSE-type estimator of the transformed linear-drift parameter are
\[
     \widehat{\beta}_{\varepsilon,n}^{X,\circ}:=\mathcal{B}_{\varepsilon,n}(Z^{X,\varepsilon,\circ}),\qquad\widehat{\beta}_{\varepsilon,n}^{Y,\circ}:=\mathcal{B}_{\varepsilon,n}
    (Z^{Y,\varepsilon,\circ}),\qquad\widehat{\beta}_{\varepsilon,n}=\mathcal{B}_{\varepsilon,n}(\widehat{Z}^{\varepsilon}).
\]

\section{Main Asymptotic Result}\label{sec3}
\begin{theorem}[Consistency and first-order asymptotic distribution]
\label{thm:asymptotic-distribution}
Suppose that Assumptions~\ref{ass:A1},~\ref{ass:A2},~\ref{ass:A3},~\ref{ass:A4},~\ref{ass:A5} hold. Then the feasible two-step estimator \eqref{final-elasticity-estimator} is consistent,
\[
    \widehat{\vartheta}_{\varepsilon,n}
    \xrightarrow{\mathbb{P}\text{-prob.}}\vartheta_{0},
\]
and has the first-order limit
\begin{equation}\label{elasticity-limit-variable}
    \varepsilon^{-1}\big(\widehat{\vartheta}_{\varepsilon,n}-\vartheta_{0}\big)\xrightarrow{\mathbb{P}\text{-prob.}}\xi_{\vartheta_{0}}, \quad\text{where}\quad\xi_{\vartheta_{0}}:=-\frac{\nu(1-\vartheta_{0})}{\mu}\frac{\int_{0}^{T}\mathcal{X}_{t}^{3-3\vartheta_{0}}dW_{t}}{\int_{0}^{T}\mathcal{X}_{t}^{4-4\vartheta_{0}}dt}, \quad\mathcal{X}_{t}=x_{0}e^{\mu t}.
\end{equation}
\begin{equation}\label{elasticity-asymptotic-variance}
    \mathcal{L}_{\mathbb{P}}(\xi_{\vartheta_{0}})=\mathcal{N}(0,V_{\vartheta_{0}}),\quad\text{with}\quad V_{\vartheta_{0}}:=\frac{\nu^{2}(1-\vartheta_{0})^{2}\int_{0}^{T}\mathcal{X}_{t}^{6-6\vartheta_{0}}dt}{\mu^{2}\big(\int_{0}^{T}\mathcal{X}_{t}^{4-4\vartheta_{0}}dt\big)^{2}}=\frac{8\nu^{2}(1-\vartheta_{0})^{3}}{-3\mu x_{0}^{2(1-\vartheta_{0})}}\frac{1-e^{6\mu(1-\vartheta_{0})T}}{\big[1-e^{4\mu(1-\vartheta_{0})T}\big]^{2}}.
\end{equation}
\end{theorem}
\subsection{Proof of the main theorem}
\begin{remark}[Structure of the proof]\label{rem:estimator-hierarchy}
The argument proceeds from the auxiliary CIR oracle under $\mathbb{Q}^{\varepsilon}$ to the same estimator under $\mathbb{P}$, then replaces the auxiliary CKLS path by the observed CEV path and the unknown elasticity $\vartheta_{0}$ by its preliminary estimate. Appendix~\ref{aaa} justifies the measure transfer, and Appendix~\ref{bbb} proves the oracle limit and the two approximation bounds: $\widehat{\beta}_{\varepsilon,n}^{Y,\circ}
    \text{ under }\mathbb{Q}^{\varepsilon}\to
    \widehat{\beta}_{\varepsilon,n}^{Y,\circ}
    \text{ under }\mathbb{P}\to
    \widehat{\beta}_{\varepsilon,n}^{X,\circ}
    \text{ under }\mathbb{P}\to
    \widehat{\beta}_{\varepsilon,n}
    \text{ under }\mathbb{P}$.
\end{remark}
\textit{Proof of Theorem~\ref{thm:asymptotic-distribution}.} By Theorem~\ref{thm:auxiliary-oracle-lse}, the auxiliary CIR oracle has a first-order limit under $\mathbb{Q}^{\varepsilon}$. The local-to-CEV scaling
gives total-variation closeness between $\mathbb{Q}^{\varepsilon}$ and $\mathbb{P}$ (Theorem~\ref{thm:total-variation-closeness}). Applying Theorem~\ref{thm:transfer-convergence} therefore yields
\begin{equation}\label{auxiliary-oracle-limit-main-proof}
    \varepsilon^{-1}\big(\widehat{\beta}_{\varepsilon,n}^{Y,\circ}-\beta_{0}\big)\xrightarrow{\mathbb{P}\text{-prob.}}\xi_{\beta_{0}},\qquad\text{where}\quad\xi_{\beta_{0}}=-2\nu(1-\vartheta_{0})\frac{\int_{0}^{T}\mathcal{X}_{t}^{3-3\vartheta_{0}}dW_{t}}{\int_{0}^{T}\mathcal{X}_{t}^{4-4\vartheta_{0}}dt}.
\end{equation}
Here Assumption~\ref{ass:A5} implies the required oracle discretization condition $\Delta_{n}=o(\varepsilon)$.

Next, Theorem~\ref{thm:auxiliary-cev-oracle-equivalence} gives the path-replacement bound
\begin{equation}\label{auxiliary-cev-equivalence-main-proof}
    \widehat{\beta}_{\varepsilon,n}^{X,\circ}-\widehat{\beta}_{\varepsilon,n}^{Y,\circ}=o_{\mathbb{P}}(\varepsilon),
\end{equation}
For the plug-in step, Theorem~\ref{thm:preliminary-estimator-rate} and Assumptions~\ref{ass:A4},~\ref{ass:A5} give
\[
	\widetilde{\vartheta}_{\varepsilon,n}-\vartheta_{0}=o_{\mathbb{P}}(\varepsilon).
\]
Together with the uniform parameter-sensitivity bound of Theorem~\ref{thm:plugin-equivalence}, this implies
\begin{equation}\label{plugin-equivalence-main-proof}
    \widehat{\beta}_{\varepsilon,n}-\widehat{\beta}_{\varepsilon,n}^{X,\circ}=o_{\mathbb{P}}(\varepsilon).
\end{equation}
Combining \eqref{auxiliary-oracle-limit-main-proof}, \eqref{auxiliary-cev-equivalence-main-proof}, and \eqref{plugin-equivalence-main-proof}, 
\[
    \varepsilon^{-1}\big(\widehat{\beta}_{\varepsilon,n}-\beta_{0}\big)=\varepsilon^{-1}\big(\widehat{\beta}_{\varepsilon,n}-\widehat{\beta}_{\varepsilon,n}^{X,\circ}\big)+\varepsilon^{-1}\big(\widehat{\beta}_{\varepsilon,n}^{X,\circ}-\widehat{\beta}_{\varepsilon,n}^{Y,\circ}\big)+\varepsilon^{-1}\big(\widehat{\beta}_{\varepsilon,n}^{Y,\circ}-\beta_{0}\big)\xrightarrow{\mathbb{P}\text{-prob.}}\xi_{\beta_{0}},
\]
and using \eqref{final-elasticity-estimator} and \eqref{elasticity-beta-relation}, we obtain $\varepsilon^{-1}(\widehat{\vartheta}_{\varepsilon,n}-\vartheta_{0})=(2\mu)^{-1}\varepsilon^{-1}(\widehat{\beta}_{\varepsilon,n}-\beta_{0})\xrightarrow{\mathbb{P}\text{-prob.}}\xi_{\vartheta_{0}}$. Consistency follows because $\varepsilon\to0$. The integrand in \eqref{elasticity-limit-variable} is deterministic, so the limit is centered Gaussian with variance given by  \eqref{elasticity-asymptotic-variance}. \hfill\qedsymbol

The preliminary estimator is $o_{\mathbb{P}}(\varepsilon)$-accurate under the stated assumptions. The second step supplies the transformation-based Gaussian limit above; it is not asserted to improve the preliminary estimator's convergence rate.

\begin{remark}[Direct CEV mapping and the zero boundary] 
The CKLS--CIR construction involves inverse powers of the state in both the transformed dynamics and the Girsanov kernel, requiring control near $0$ when transferring asymptotic results to the absorbing CEV model with $1/2<\vartheta<1$. Define $\tau_{0}^{\varepsilon}:=\inf\{t\geq0:X_{t}^{\varepsilon,\vartheta_{0}}=0\}$. After localization away from $0$, It\^{o}'s formula applied to $Z_{t}^{X,\varepsilon,\circ}=\mathcal{T}(X_{t}^{\varepsilon,\vartheta_{0}};\vartheta_{0})$ gives, for $t<\tau_{0}^{\varepsilon}$,
\[
    dZ_{t}^{X,\varepsilon,\circ}=\big(\alpha_{\varepsilon}^{\mathrm{CEV}}-\beta_{0}Z_{t}^{X,\varepsilon,\circ}\big)dt+\varepsilon\nu\zeta\sqrt{Z_{t}^{X,\varepsilon,\circ}}dW_{t},\qquad\text{where}\quad    \alpha_{\varepsilon}^{\mathrm{CEV}}=\frac{\varepsilon^{2}\nu^{2}\zeta^{2}(1-2\vartheta_{0})}{4(1-\vartheta_{0})}<0,
\]
and $\beta_{0}=-2\mu(1-\vartheta_{0})>0$. As noted in the property \eqref{prop:CEV-CIR}, this CIR-type identity has negative constant-drift and does not provide the strictly positive CIR benchmark required by our measure-change construction. It cannot hold unchanged beyond $\tau_{0}^{\varepsilon}$: the mapped process remains at $0$ after absorption, whereas the displayed drift at $0$ is $\alpha_{\varepsilon}^{\mathrm{CEV}}<0$. The absorbed continuation therefore requires a stopped formulation. Moreover, an equivalent change of measure cannot eliminate a positive-probability event on which the same mapped process hits $0$. The auxiliary construction in Subsection~\ref{subsec:local-to-cev} instead applies the CKLS--CIR transform of Appendix~\ref{aaa} to a strictly positive auxiliary process, providing the strictly positive CIR benchmark for localized oracle asymptotics.
\end{remark}

\section{Simulation Study}\label{sec4}
We examine the Gaussian approximation of the CEV-based oracle estimator, the finite-sample cost of the preliminary plug-in step, and sensitivity to the elasticity, local window, and numerical scheme. All estimates use simulated observations of the target CEV process; the oracle uses the true elasticity only in the state mapping.

\subsection{Design and diagnostics}\label{subsec:simulation-design}
We simulate the target CEV model $dX_{t}^{\varepsilon,\vartheta_{0}}=\mu X_{t}^{\varepsilon,\vartheta_{0}}dt+\varepsilon\nu(X_{t}^{\varepsilon,\vartheta_{0}})^{\vartheta_{0}}dW_{t}$ with $\mu=-0.8$, $\nu=0.7$, $\vartheta_{0}=0.70$, $X_{0}^{\varepsilon,\vartheta_{0}}=x_{0}=2$, and $T=1$. The preliminary estimator in \eqref{preliminary-estimator} uses $\tau=0.20$, $\Theta=[0.55,0.95]$, $\vartheta^{\diamond}=(0.55+0.95)/2=0.75$, and the exact deterministic within-window correction $\Phi_{n}(\vartheta)=(e^{2\mu\vartheta h_{n}}-1)/(2\mu\vartheta h_{n})$. We set $\zeta=1$ in the state mapping and $c_{0}=0.9|\log\mathcal{X}_{\tau}|$, where $\mathcal{X}_{\tau}=2e^{-0.16}\approx1.7043$ and $|\log\mathcal{X}_{\tau}|\approx0.5331$, ensuring local identification. Observations are sampled at $t_{i}=i\Delta_{n}$, with $\Delta_{n}=1/n$. We take $m_{n}$ as the nearest integer to $\Delta_{n}^{-2/3}$ and set $h_{n}=m_{n}\Delta_{n}$ and $\tau_{n}=\tau$ in all main designs. The quantity $\rho_{\varepsilon,n}:=\sqrt{\Delta_{n}/h_{n}}/\varepsilon\asymp\Delta_{n}^{1/3}/\varepsilon$ summarizes the joint restrictions in Assumption~\ref{ass:A5}: with this window choice, $\rho_{\varepsilon,n}\to0$ is equivalent to $\Delta_{n}=o(\varepsilon^{3})$. Table~\ref{tab:main-asymptotic-diagnostics} reports $7$ designs, each with $R=3{,}000$ replications. Paths are generated using a Milstein scheme with nonnegative projection and two internal steps per observation interval; once a simulated path reaches $0$, it remains there.

Let $\widehat{\vartheta}_{\varepsilon,n}^{X,\circ}:=1+\widehat{\beta}_{\varepsilon,n}^{X,\circ}/(2\mu)$ denote the CEV-based oracle estimator and $\widehat{\vartheta}_{\varepsilon,n}$ the feasible estimator in \eqref{final-elasticity-estimator}. The oracle drift estimate $\widehat{\beta}_{\varepsilon,n}^{X,\circ}$ is computed by applying the LSE functional \eqref{path-functional-B} to $Z_{t_{i}}^{X,\varepsilon,\circ}=\mathcal{T}(X_{t_{i}}^{\varepsilon,\vartheta_{0}};\vartheta_{0})$, using the same simulated CEV observations and the true elasticity in the state mapping. We report
\[
    S_{\varepsilon,n}^{\circ}:=\frac{\widehat{\vartheta}_{\varepsilon,n}^{X,\circ}-\vartheta_{0}}{\varepsilon\sqrt{V_{\vartheta_{0}}}},\qquad S_{\varepsilon,n}:=\frac{\widehat{\vartheta}_{\varepsilon,n}-\vartheta_{0}}{\varepsilon\sqrt{V_{\vartheta_{0}}}},\qquad D_{\varepsilon,n}^{\mathrm{plug}}:=\frac{\widehat{\vartheta}_{\varepsilon,n}-\widehat{\vartheta}_{\varepsilon,n}^{X,\circ}}{\varepsilon},
\]
where $V_{\vartheta_{0}}=\nu^{2}(1-\vartheta_{0})^{2}\int_{0}^{T}\mathcal{X}_{t}^{6-6\vartheta_{0}}dt/[\mu^{2}(\int_{0}^{T}\mathcal{X}_{t}^{4-4\vartheta_{0}}dt)^{2}]\approx0.058300$ at the baseline parameters. We report the means and standard deviations of $S^{\circ}$ and $S$, the empirical frequencies of $|S^{\circ}|\leq1.96$ and $|S|\leq1.96$, and the empirical root mean square of $D^{\mathrm{plug}}$. These coverage diagnostics use the true-parameter variance to assess the Gaussian approximation. To assess feasible inference, define
\[
    \widehat{V}_{\varepsilon,n}:=V_{\widetilde{\vartheta}_{\varepsilon,n}},\qquad S_{\varepsilon,n}^{\mathrm{st}}:=\frac{\widehat{\vartheta}_{\varepsilon,n}-\vartheta_{0}}{\varepsilon\sqrt{\widehat{V}_{\varepsilon,n}}},\qquad I_{\varepsilon,n}:=\big[\widehat{\vartheta}_{\varepsilon,n}-z_{0.975}\varepsilon\sqrt{\widehat{V}_{\varepsilon,n}},~\widehat{\vartheta}_{\varepsilon,n}+z_{0.975}\varepsilon\sqrt{\widehat{V}_{\varepsilon,n}}\big].
\]
Here $z_{0.975}$ is the standard normal $0.975$ quantile, so $I_{\varepsilon,n}$ is the nominal $95\%$ Wald interval. The function $V_\vartheta$ is the same deterministic-path variance evaluated at $\vartheta$, with the remaining model parameters treated as known. Since $\widetilde{\vartheta}_{\varepsilon,n}\xrightarrow{\mathbb{P}}\vartheta_{0}$ and $\vartheta\mapsto V_\vartheta$ is continuous and positive on $\Theta$, we have $\widehat{V}_{\varepsilon,n}\xrightarrow{\mathbb{P}}V_{\vartheta_{0}}>0$. Theorem~\ref{thm:asymptotic-distribution} and Slutsky's theorem then yield $S_{\varepsilon,n}^{\mathrm{st}}\xrightarrow{d}\mathcal{N}(0,1)$. For each of the $7$ main designs, we report the mean and standard deviation of $S_{\varepsilon,n}^{\mathrm{st}}$, interval coverage with its Monte Carlo standard error, and mean interval length. Intervals are not truncated to $\Theta$, and all replications, including boundary solutions of the preliminary estimator, are retained.

\subsection{Results of main simulation study: estimation of the CEV elasticity parameter}\label{subsec:main-simulation-results}
The oracle approximation is accurate throughout the $7$ designs: standardized standard deviations range from $1.001$ to $1.022$ and coverage from $94.00\%$ to $95.10\%$ (Table~\ref{tab:main-asymptotic-diagnostics}). The feasible estimator's larger departure from the standard normal benchmark is mainly overdispersion. As $\rho_{\varepsilon,n}$ decreases from $0.2321$ to $0.0489$, its standardized standard deviation falls from $1.825$ to $1.035$ and coverage rises from $72.43\%$ to $94.07\%$. Standardized means are small in magnitude but do not approach $0$ monotonically.

\begin{table}[!htbp]
\centering
\caption{Main designs and standardized estimation diagnostics ($R=3{,}000$).}
\label{tab:main-asymptotic-diagnostics}
\small
\setlength{\tabcolsep}{3.5pt}
\begin{tabular}{crrrrrrrrr}
\toprule
& & & \multicolumn{3}{c}{Oracle $S^{\circ}$} & \multicolumn{3}{c}{Feasible $S$} & \\
\cmidrule(lr){4-6}\cmidrule(lr){7-9}
$\varepsilon$ & $n$ & $\rho_{\varepsilon,n}$ & Mean & SD & Cov.\ (\%) & Mean & SD & Cov.\ (\%) & $\mathrm{RMSE}(D^{\mathrm{plug}})$\\
\midrule
$0.200$ & $10,000$ & $0.2321$ & $-0.133$ & $1.002$ & $94.70$ & $-0.112$ & 1.825 & $72.43$ & $0.3534$ \\
$0.150$ & $50,000$ & $0.1810$ & $-0.122$ & $1.022$ & $94.00$ & $-0.145$ & $1.508$ & $81.50$ & $0.2611$ \\
$0.100$ & $250,000$ & $0.1587$ & $-0.050$ & $1.014$ & $94.23$ & $-0.025$ & $1.403$ & $83.63$ & $0.2257$ \\
$0.075$ & $1,000,000$ & $0.1333$ & $-0.066$ & $1.010$ & $94.73$ & $-0.071$ & $1.291$ & $87.37$ & $0.1860$ \\
$0.050$ & $10,000,000$ & $0.0928$ & $-0.027$ & $1.003$ & $95.10$ & $-0.016$ & $1.137$ & $91.50$ & $0.1285$ \\
$0.040$ & $50,000,000$ & $0.0679$ & $-0.009$ & $1.019$ & $94.90$ & $-0.010$ & $1.097$ & $92.67$ & $0.0956$ \\
$0.035$ & $200,000,000$ & $0.0489$ & $-0.055$ & $1.001$ & $94.87$ & $-0.053$ & $1.035$ & $94.07$ & $0.0679$ \\
\bottomrule
\end{tabular}
\par\smallskip
\begin{minipage}{0.90\textwidth}
\footnotesize
\emph{Notes.} $\Delta_{n}=1/n$, $m_{n}=\operatorname{round}(\Delta_{n}^{-2/3})$, and $h_{n}=m_{n}\Delta_{n}$. Coverage is evaluated at the nominal $95\%$ level using $V_{\vartheta_{0}}$. Its Monte Carlo standard error is about $0.4$ percentage points near $95\%$, and $0.82$ points for the first feasible row.
\end{minipage}
\end{table}

At the finest design $(\varepsilon,n)=(0.035,2\times10^{8})$, the empirical densities and normal Q--Q plots in Figure~\ref{fig:standardized-distributions} closely match the standard normal benchmark, with mild overdispersion and tail departures for the feasible estimator. The standardized standard deviations are $1.001$ and $1.035$ for the oracle and feasible estimators, respectively.

\begin{figure}[!htbp]
\centering
\begin{subfigure}[t]{0.42\textwidth}
    \centering
    \includegraphics[width=\linewidth]{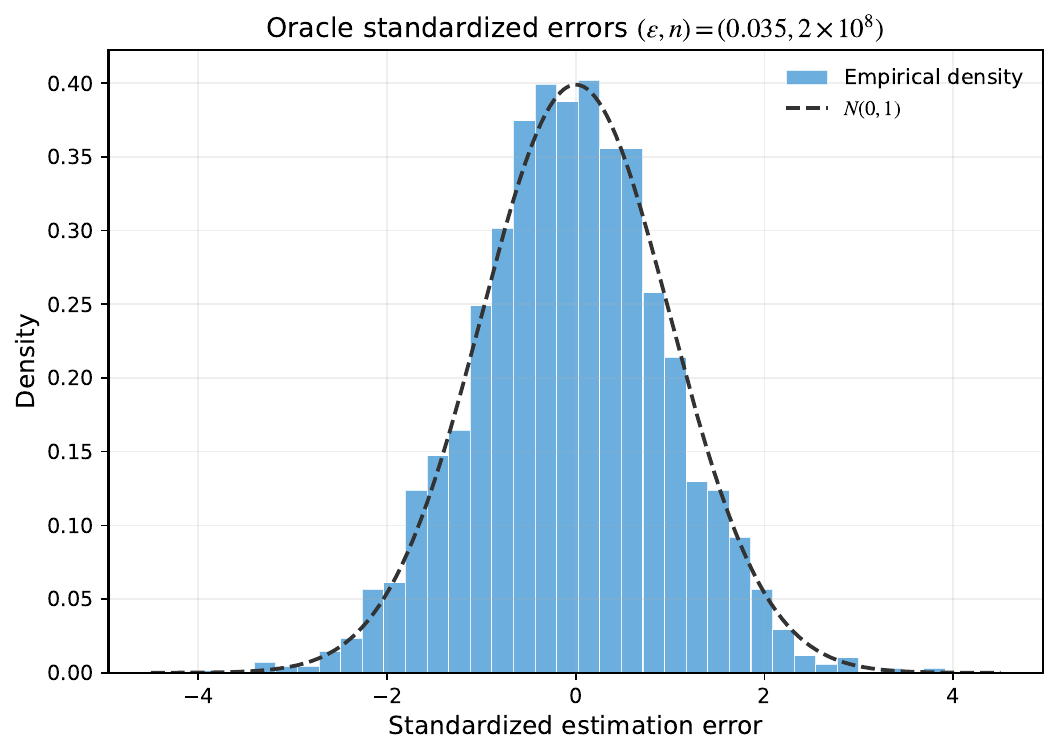}
    \caption{Oracle density.}
    \label{fig:oracle-hist}
\end{subfigure}
\hfill
\begin{subfigure}[t]{0.42\textwidth}
    \centering
    \includegraphics[width=\linewidth]{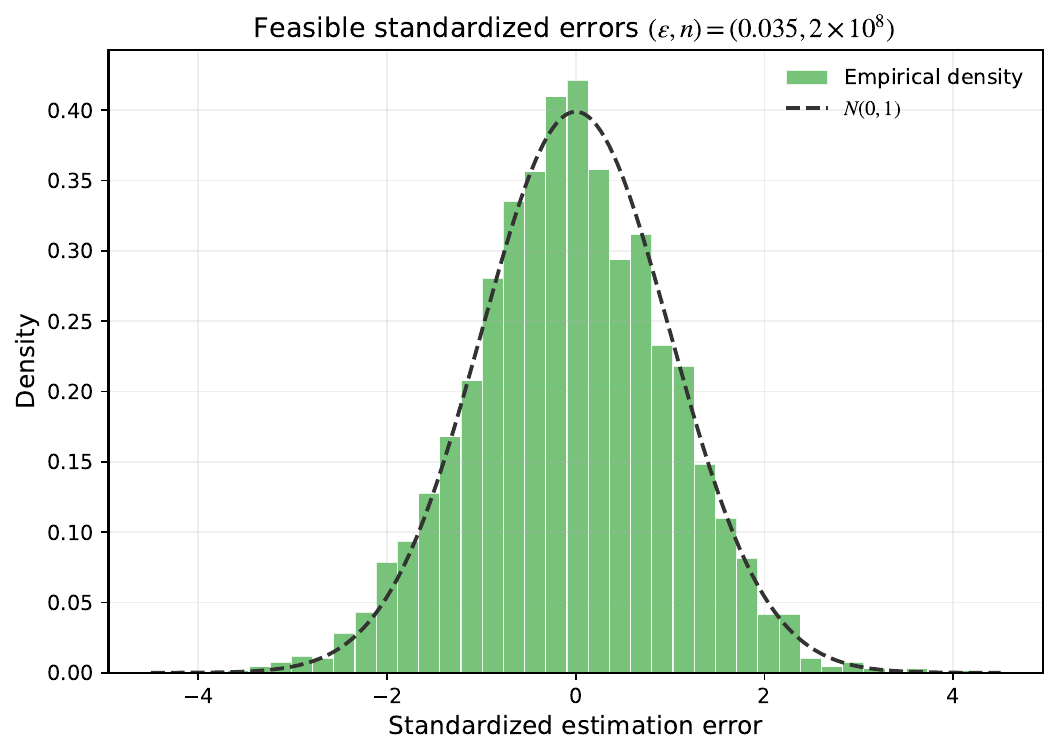}
    \caption{Feasible density.}
    \label{fig:feasible-hist}
\end{subfigure}

\begin{subfigure}[t]{0.42\textwidth}
    \centering
    \includegraphics[width=\linewidth]{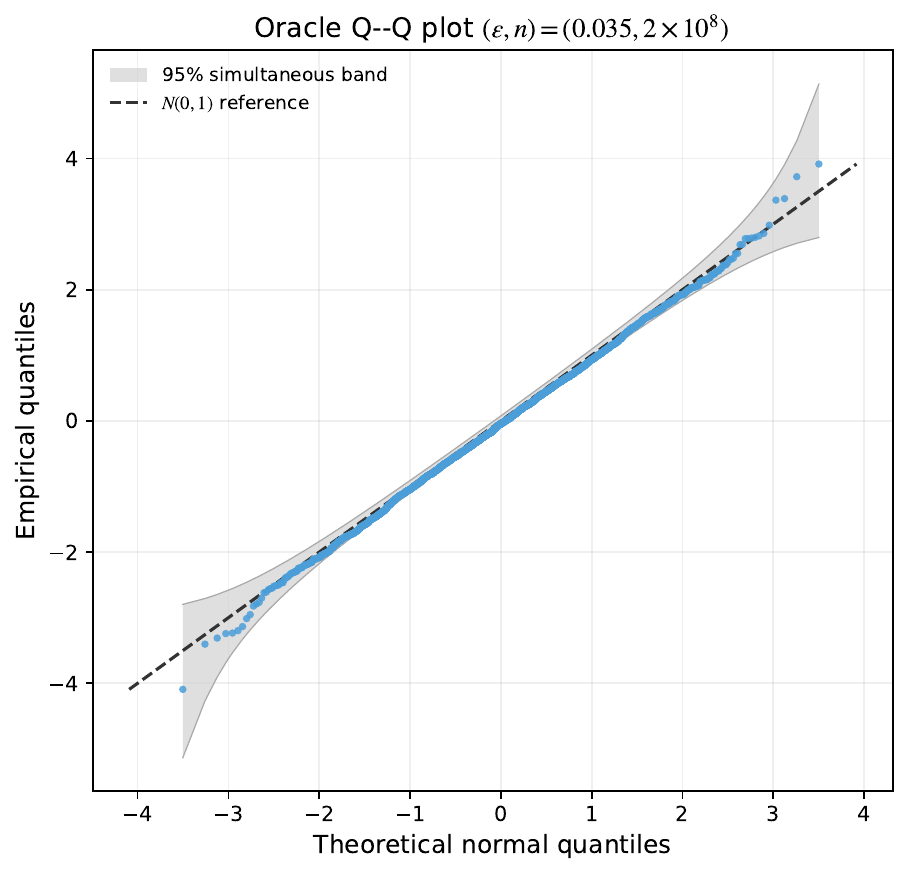}
    \caption{Oracle Q--Q plot.}
    \label{fig:oracle-qq}
\end{subfigure}
\hfill
\begin{subfigure}[t]{0.42\textwidth}
    \centering
    \includegraphics[width=\linewidth]{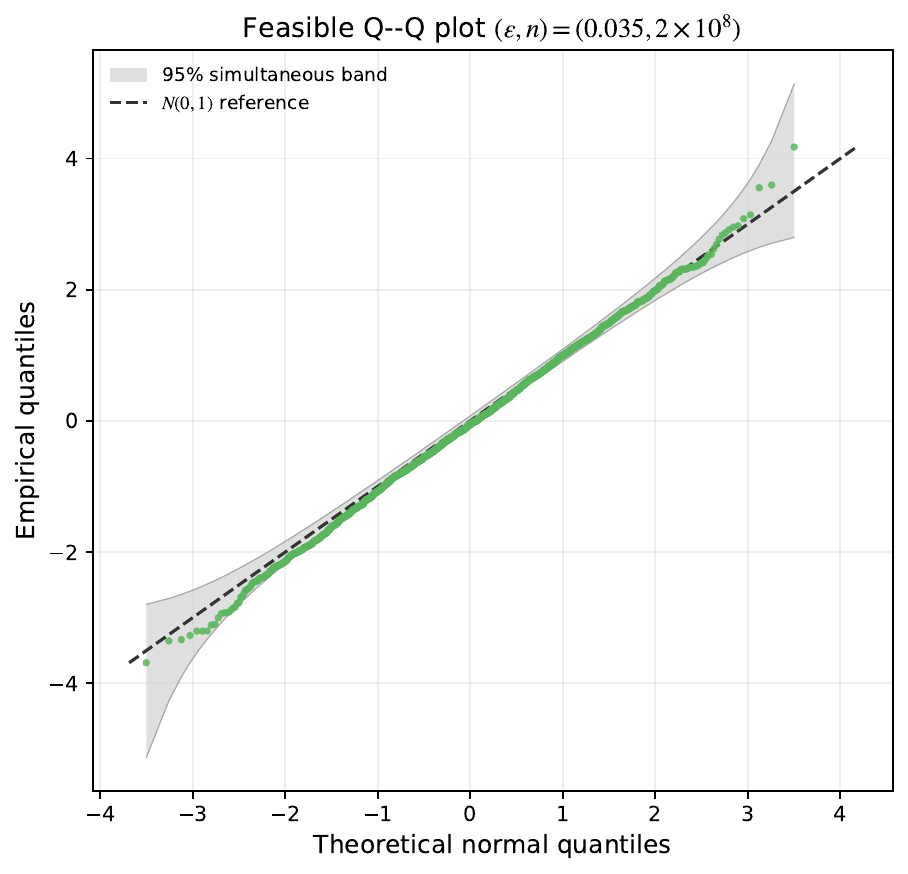}
    \caption{Feasible Q--Q plot.}
    \label{fig:feasible-qq}
\end{subfigure}
\caption{Standardized estimation errors at $(\varepsilon,n)=(0.035,2\times10^{8})$. The left and right halves correspond to
$\widehat{\vartheta}^{X,\circ}_{\varepsilon,n}$ and
$\widehat{\vartheta}_{\varepsilon,n}$, respectively. Dashed references show the $\mathcal{N}(0,1)$ density in (a)--(b) and $y=x$ in (c)--(d). Shading denotes 95\% simultaneous $\mathcal{N}(0,1)$ reference bands within each Q--Q panel, obtained using a two-sided equal-local-level construction calibrated from $20{,}000$ Monte Carlo samples of size $3{,}000$.}
\label{fig:standardized-distributions}
\end{figure}

\begin{figure}[!htbp]
\centering
\includegraphics[width=0.68\textwidth]{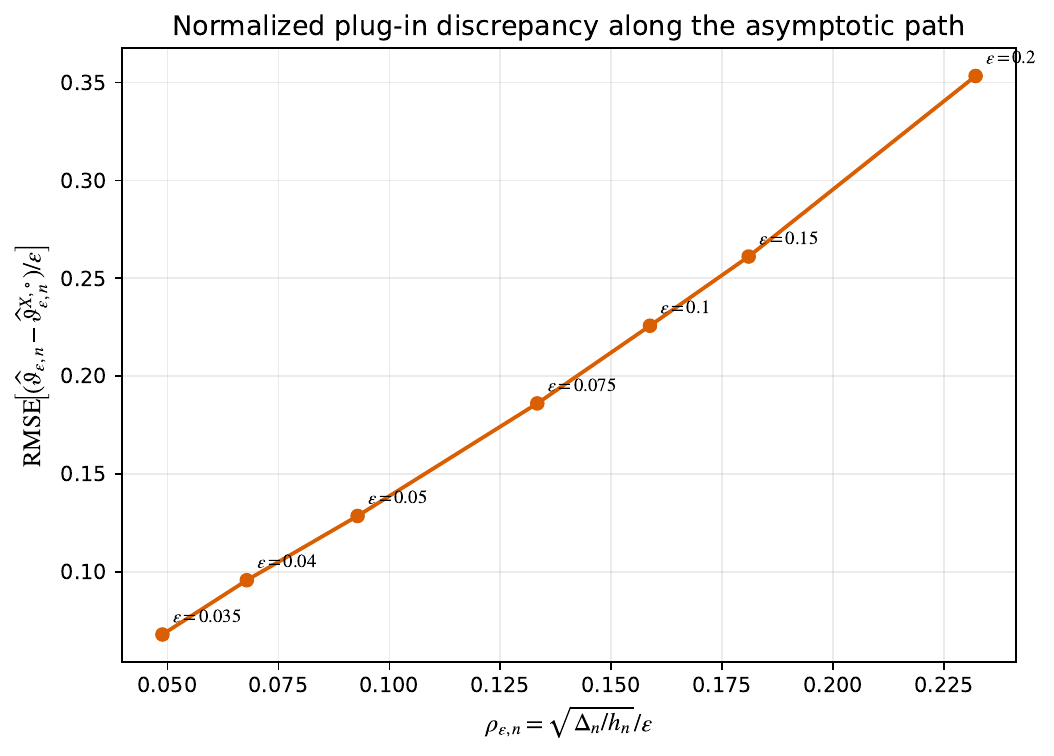}
\caption{Normalized plug-in RMSE against
$\rho_{\varepsilon,n}$ for the $7$ main designs. Points use the unrounded Monte Carlo RMSE estimates underlying Table~\ref{tab:main-asymptotic-diagnostics}.}
\label{fig:plugin-gap}
\end{figure}
Figure~\ref{fig:plugin-gap} shows that the normalized plug-in RMSE decreases from $0.3534$ to $0.0679$ as $\rho_{\varepsilon,n}$ decreases from $0.2321$ to $0.0489$. Across the $7$ designs, $\mathrm{RMSE}(D^{\mathrm{plug}})/\rho_{\varepsilon,n}$ ranges from approximately $1.384$ to $1.522$, while the gap between the oracle and feasible standardized standard deviations also narrows. These patterns are consistent with a diminishing plug-in effect along the joint asymptotic path. For the corrected preliminary estimator, the maximum absolute Monte Carlo bias is approximately $1.95\times10^{-3}$, falling below $2\times10^{-5}$ in the two finest designs. The event $\mathcal{E}_{\varepsilon,n}$ occurred in every main replication, so the fallback was unused. The minimizer over $\Theta$ reached the lower boundary in $46/3{,}000$ replications ($1.53\%$) in the coarsest design and in one replication at $(\varepsilon,n)=(0.15,50{,}000)$, with no other boundary solutions observed. All replications, including boundary solutions, were retained in the reported summaries, which therefore include the effects of the parameter constraints.

\begin{table}[!htbp]
\centering
\small
\setlength{\tabcolsep}{4pt}
\caption{Feasible studentization and nominal 95\% Wald intervals
($R=3{,}000$).}
\label{tab:studentized-diagnostics}
\begin{tabular}{@{}rrrrrrrr@{}}
\toprule
$\varepsilon$
& $n$
& $\mathrm{Mean}(S^{\mathrm{st}})$
& $\mathrm{SD}(S^{\mathrm{st}})$
& \shortstack{Cov. (\%)\\$V_{\vartheta_{0}}$}
& \shortstack{Cov. (\%)\\$\widehat{V}_{\varepsilon,n}$}
& \shortstack{MCSE\\(pp)}
& \shortstack{Mean\\length}\\
\midrule
$0.200$ & $10,000$ & $0.254$ & $2.121$ & $72.43$ & $71.87$ & $0.82$ & $0.18747$ \\
$0.150$ & $50,000$ & $-0.013$ & $1.536$ & $81.50$ & $80.37$ & $0.73$ & $0.14206$ \\
$0.100$ & $250,000$ & $0.041$ & $1.409$ & $83.63$ & $83.60$ & $0.68$ & $0.09441$ \\
$0.075$ & $1,000,000$ & $-0.037$ & $1.291$ & $87.37$ & $87.17$ & $0.61$ & $0.07099$ \\
$0.050$ & $10,000,000$ & $-0.005$ & $1.138$ & $91.50$ & $91.53$ & $0.51$ & $0.04730$ \\
$0.040$ & $50,000,000$ & $-0.005$ & $1.096$ & $92.67$ & $92.73$ & $0.47$ & $0.03786$ \\
$0.035$ & $200,000,000$ & $-0.051$ & $1.034$ & $94.07$ & $94.10$ & $0.43$ & $0.03312$ \\
\bottomrule
\end{tabular}

\par\smallskip
\begin{minipage}{\linewidth}
\footnotesize
\textit{Notes.} Both coverage columns concern $\widehat{\vartheta}_{\varepsilon,n}$, using $V_{\vartheta_{0}}$ or $\widehat{V}_{\varepsilon,n}$. For the latter, MCSE denotes the Monte Carlo standard error of coverage (percentage points), and mean interval length is on the elasticity scale. Intervals use the standard normal $0.975$ quantile without truncation to $\Theta$. All $3{,}000$ replications per design are retained, including boundary solutions of the preliminary estimator.
\end{minipage}
\end{table}

Table \ref{tab:studentized-diagnostics} reports the results obtained using the estimated variance $\widehat{V}_{\varepsilon,n}$. Along the $7$ main designs, the standard deviation of the studentized error decreases from $2.121$ to $1.034$, while interval coverage increases from $71.87\%$ to $94.10\%$ and mean interval length decreases from $0.18747$ to $0.03312$. For the five designs with $\varepsilon\leq0.10$, coverage differs from its true-variance counterpart by at most $0.20$ percentage points. At the finest design, coverage is $94.10\%$, compared with $94.07\%$ using $V_{\vartheta_{0}}$, with a Monte Carlo standard error of $0.43$ percentage points. These results support the feasibility of variance estimation for interval inference along the joint asymptotic sequence, although undercoverage remains pronounced in the coarser designs.

\subsection{Results of the additional simulation studies: sensitivity and numerical robustness}\label{subsec:simulation-robustness}
We examine the sensitivity of the finite-sample results to elasticity, local-window width, and numerical approximation. These comparisons help distinguish changes associated with the model parameter and preliminary estimation from those associated with the simulation scheme. Three additional studies use separate reproducible random-number streams, with common Brownian increments within each study. Their baseline results hence need not equal those of the main experiment. For the window comparison, set $m_{n}=\operatorname{round}(c_{h}\Delta_{n}^{-2/3})$, with baseline $c_{h}=1$. Across these studies, the exceptional-event fallback was unused and no numerical failures or excluded replications were recorded. The Euler comparison uses the same nonnegative projection as in the main simulation study.

\begin{table}[H]
\centering
\caption{Sensitivity to elasticity, local-window width, and numerical scheme.}
\label{tab:sensitivity-condensed}
\small
\setlength{\tabcolsep}{3.5pt}
\renewcommand{\arraystretch}{1.08}

\begin{tabular}{@{}c|cccccccc@{}}
\toprule
Study & Setting & $\rho_{\varepsilon,n}$ & $V_{\vartheta_{0}}$ & $\mathrm{SD}(\widetilde{\vartheta})$ & $\mathrm{SD}(S^{\circ})$ & $\mathrm{SD}(S)$ & Cov.\ (\%) & $\mathrm{RMSE}(D^{\mathrm{plug}})$\\
\hline

\multirow{3}{*}{%
  \shortstack[c]{%
    \emph{A. Elasticity}\\
    $R=1{,}500$\\
    $c_h=1$}}
& $\vartheta_{0}=0.60$ & $0.1587$ & $0.09830$ & $0.02128$
& $1.025$ & $1.242$ & $88.13$ & $0.2175$\\
& $\vartheta_{0}=0.70$ & $0.1587$ & $0.05830$ & $0.02170$
& $1.022$ & $1.387$ & $85.07$ & $0.2216$\\
& $\vartheta_{0}=0.80$ & $0.1587$ & $0.02735$ & $0.02206$
& $1.019$ & $1.725$ & $75.93$ & $0.2250$\\
\hline

\multirow{3}{*}{%
  \shortstack[c]{%
    \emph{B. Window width}\\
    $R=1{,}200$\\
    $\vartheta_{0}=0.70$}}
& $c_h=2^{-2/3}$ & $0.2000$ & $0.05830$ & $0.02643$
& $0.995$ & $1.535$ & $79.83$ & $0.2698$\\
& $c_h=1$ & $0.1587$ & $0.05830$ & $0.02108$
& $0.995$ & $1.357$ & $85.42$ & $0.2147$\\
& $c_h=1.5$ & $0.1296$ & $0.05830$ & $0.01869$
& $0.995$ & $1.282$ & $87.33$ & $0.1898$\\
\hline

Study & Scheme & $r$
& \multicolumn{3}{c}{Oracle paired RMSE}
& \multicolumn{3}{c@{}}{Feasible paired RMSE}\\
\hline

\multirow{3}{*}{%
  \shortstack[c]{%
    \emph{C. Numerics}\\
    $R=1{,}000$\\
    $\vartheta_{0}=0.70$}}
& Milstein & $1$ & \multicolumn{3}{c}{$3.498\times10^{-5}$} & \multicolumn{3}{c@{}}{$3.406\times10^{-4}$}\\ 
& Milstein & $4$ & \multicolumn{3}{c}{$1.747\times10^{-5}$} & \multicolumn{3}{c@{}}{$1.703\times10^{-4}$}\\
& Euler & $2$ & \multicolumn{3}{c}{$1.494\times10^{-4}$} & \multicolumn{3}{c@{}}{$4.041\times10^{-4}$}\\
\bottomrule
\end{tabular}

\par\smallskip
\begin{minipage}{0.83\textwidth}
\footnotesize
\emph{Notes.} Panels A and B use $(\varepsilon,n)=(0.10,250{,}000)$ and Milstein with two internal steps. Set $m_{n}=\operatorname{round}(c_h\Delta_{n}^{-2/3})$ and $h_{n}=m_{n}\Delta_{n}$; $\Phi_{n}$ uses the actual $h_{n}$. $\mathrm{SD}(\widetilde{\vartheta})$ is reported on the unstandardized elasticity scale. Coverage refers to the feasible estimator at the nominal $95\%$ level using the corresponding true-parameter variance. Panel B uses identical observations for all windows, so its oracle estimates coincide replication by replication. Panel C uses $(\varepsilon,n)=(0.15,50{,}000)$, where $r$ denotes the number of internal steps per observation interval. Paired RMSEs are computed relative to baseline Milstein with $r=2$, with differences normalized by $\varepsilon\sqrt{V_{\vartheta_{0}}}$. Feasible coverage is $80.00\%$ under all four numerical specifications.
\end{minipage}
\end{table}

\textbf{Elasticity.} At $(\varepsilon,n)=(0.10,250{,}000)$, increasing $\vartheta_{0}$ from $0.60$ to $0.80$ leaves oracle standardized dispersion nearly unchanged but increases feasible standardized dispersion and reduces coverage (Table~\ref{tab:sensitivity-condensed}, panel A). The plug-in RMSE changes only from $0.2175$ to $0.2250$, while $V_{\vartheta_{0}}$ decreases from $0.09830$ to $0.02735$. Since $S-S^{\circ}=D^{\mathrm{plug}}/\sqrt{V_{\vartheta_{0}}}$, a similar discrepancy on the $\varepsilon$ scale becomes larger relative to the oracle asymptotic standard deviation. This concerns standardized accuracy: the unstandardized feasible standard deviation decreases from $0.03894$ to $0.02854$. Meanwhile, the preliminary standard deviation changes only slightly, from $0.02128$ to $0.02206$, indicating relatively stable dispersion of the preliminary estimator on the elasticity scale. Preliminary biases remain below their Monte Carlo standard errors.

\textbf{Local window.}
At the same design with $\vartheta_{0}=0.70$, enlarging $c_{h}$ from $2^{-2/3}$ to $1.5$ reduces preliminary standard deviation from $0.02643$ to $0.01869$, while absolute bias remains below $3.8\times10^{-4}$. Feasible coverage improves from $79.83\%$ to $87.33\%$, and plug-in RMSE falls by $29.6\%$ (Table~\ref{tab:sensitivity-condensed} panel B). With identical observations across windows, oracle estimates coincide replication by replication, isolating the window's effect on preliminary estimation and the plug-in step. This is consistent with the reduction in $\sqrt{\Delta_{n}/h_{n}}$ outweighing the increase in $\varepsilon\sqrt{h_{n}}$ over the tested range in Theorem~\ref{thm:preliminary-estimator-rate}. It does not establish an optimal window, and undercoverage remains at the longest window.

\textbf{Numerical scheme.}
At $(\varepsilon,n,\vartheta_{0})=(0.15,50{,}000,0.70)$, $1{,}000$ paired replications compare Milstein with $r=1,2,4$ internal steps and Euler with $r=2$, using a common Brownian driver on the mesh $\Delta_{n}/4$. Relative to Milstein with $r=2$, the largest paired RMSE is approximately $4.041\times10^{-4}$ in units of $\varepsilon\sqrt{V_{\vartheta_{0}}}$ (Table~\ref{tab:sensitivity-condensed}, panel C). Feasible coverage is $80.00\%$ under every specification. The Milstein paired means are distinguishable from $0$ at Monte Carlo precision but small on the standardized scale. The tested changes in scheme and refinement have little effect on the finite-sample overdispersion at this configuration. These comparisons do not bound the error relative to the exact SDE solution.

\subsection{Summary of the Simulation Study}
Overall, the main designs support the oracle Gaussian approximation and a diminishing preliminary plug-in effect, with feasible interval coverage approaching the nominal level along the joint asymptotic sequence. The additional studies identify elasticity and window width as relevant to finite-sample standardized accuracy, while supporting numerical stability at the configuration tested.

\section*{Acknowledgements}
The second author was partially supported by JSPS KAKENHI Grant Numbers JP21K03358 and JST CREST JPMJCR14D7, Japan. The second author was supported by JSPS KAKENHI Grant Number JP16K00036.

\begin{appendix}
\makeatletter
\renewcommand{\@seccntformat}[1]{%
  \ifstrequal{#1}{section}
    {Appendix~\csname the#1\endcsname\quad}
    {\csname the#1\endcsname\quad}%
}
\makeatother

\section{The CKLS--CIR Transformation and Asymptotic Measure Transfer}\label{aaa}
For the CKLS model of Definition~\ref{def:CKLS} with $a,b,y_{0}>0$ and $k\in(1/2,1)$, this appendix develops the state mapping and measure change used in Subsection~\ref{subsec:local-to-cev}, yielding CIR-type dynamics (Definition~\ref{def:CIR}) with $\beta=2b(1-k)>0$. Under the local-to-CEV scaling, uniform inverse-moment bounds imply convergence of the Girsanov density to $1$ in $L^{1}$ and convergence of the measures in total variation, allowing asymptotic results to transfer to the original measure.

\subsection{State mapping and change of measure}\label{subsec:state-mapping}
To obtain a square-root diffusion coefficient for $Z_{t}^{(k)}:=\mathcal{T}(Y_{t};k)$, write $\zeta:=\gamma/\sigma>0$ and require $\partial_{y}\mathcal{T}(y;k)y^{k}=\zeta\sqrt{\mathcal{T}(y;k)}$. Solving this equation with zero integration constant gives
\begin{equation}\label{state-mapping-general-form}
	\mathcal{T}(y;k)=\frac{\zeta^{2}}{4(1-k)^{2}}y^{2(1-k)},\qquad y>0.
\end{equation}
Since $1-k>0$, this mapping satisfies the required square-root identity and is an increasing $C^{2}$-bijection of $(0,\infty)$ onto itself, with inverse $\mathcal{T}^{-1}(z;k)=[2(1-k)\sqrt{z}/\zeta]^{1/(1-k)}$. Its derivatives are $\partial_{y}\mathcal{T}(y;k)=\zeta^{2}y^{1-2k}/[2(1-k)]$ and $\partial_{yy}\mathcal{T}(y;k)=\zeta^{2}(1-2k)y^{-2k}/[2(1-k)]$. Using these expressions, It\^{o}'s formula, justified by localization, yields
\begin{equation}\label{mapped-dynamics-general-P}
	dZ_{t}^{(k)}=\Big[\frac{a\zeta^{2}}{2(1-k)}Y_{t}^{1-2k}-2b(1-k)Z_{t}^{(k)}+\frac{\sigma^{2}\zeta^{2}(1-2k)}{4(1-k)}\Big]dt+\sigma\zeta\sqrt{Z_{t}^{(k)}}dW_{t}.
\end{equation}
By invertibility of $\mathcal{T}$, $Z^{(k)}$ remains Markovian; denote its drift by $m_{\mathbb{P}}$. The diffusion coefficient has the CIR square-root form, but the drift is non-affine for $k\in(1/2,1)$. We retain the linear-drift coefficient $\beta=2b(1-k)$ and choose $\alpha=\sigma^{2}\zeta^{2}/2$, so that the target CIR dynamics satisfy the Feller condition with equality. To obtain this drift, define the Girsanov kernel, the shifted process and the associated stochastic exponential by
\begin{equation}\label{girsanov-kernel-general}
	q_{t}:=\frac{m_{\mathbb{P}}(Z_{t})-\alpha+\beta Z_{t}}{\sigma\zeta\sqrt{Z_{t}}}=\frac{a}{\sigma}Y_{t}^{-k}-\frac{\sigma}{2}Y_{t}^{k-1},\quad\widetilde{W}_{t}:=W_{t}+\int_{0}^{t}q_{s}ds,\quad\Lambda_{t}:=e^{-\int_{0}^{t}q_{s}dW_{s}-\frac{1}{2}\int_{0}^{t}q_{s}^{2}ds}.
\end{equation}
Since $Y$ is continuous, strictly positive and non-explosive, $\int_{0}^{T}q_{s}^{2}ds<\infty~a.s.$ for every finite $T$. Thus, $\Lambda$ is a well-defined, strictly positive local martingale. The following lemma establishes the validity of the measure change.

\begin{lemma}[Validity of the measure change]\label{lem:ckls-density-martingale}
For $k\in(1/2,1)$ and $T<\infty$, the stochastic exponential $\Lambda$ defined in \eqref{girsanov-kernel-general} is a strictly positive $\mathbb{P}$-martingale on $[0,T]$, with $\mathbb{E}^{\mathbb{P}}[\Lambda_{t}]=1$ for all $0\leq t\leq T$. Consequently, $d\mathbb{Q}/d\mathbb{P}|_{\mathcal{F}_{T}}=\Lambda_{T}$ defines an equivalent probability measure under which $\widetilde{W}$ is a standard Brownian motion on $[0,T]$.
\end{lemma}
\begin{proof}
Put $J=(0,\infty)$, $m(y)=a-by$, $\Sigma(y)=\sigma y^{k}$. Since $\Sigma>0$ and $m,\Sigma,q$ are continuous on $J$, $\Sigma^{-2},m\Sigma^{-2},q^{2}\Sigma^{-2}\in L^{1}_{\mathrm{loc}}(J)$ (Engelbert--Schmidt conditions). The auxiliary diffusion associated with the integrand $-q$, initially defined up to exit from $J$, has diffusion coefficient $\Sigma$ and drift
\[
	\check m(y):=m(y)-q(y)\Sigma(y)=\frac{\sigma^{2}}{2}y^{2k-1}-by.
\]
Its scale density, normalized at $1$, is
\[
	\check{\rho}(y)=\exp\Big\{-\int_{1}^{y}\frac{2\check m(z)}{(\Sigma(z))^{2}}dz\Big\}=\frac{1}{y}\exp\Big\{\mathcal{M}\big(y^{2(1-k)}-1\big)\Big\},\qquad\mathcal{M}:=\frac{b}{\sigma^{2}(1-k)}>0.
\]
Since $\check{\rho}(y)\geq e^{-M}/y$ for $0<y\leq1$ and $\check{\rho}(y)\geq1/y$ for $y\geq1$, its scale integrals diverge at both endpoints. Feller's explosion test therefore excludes exit of the auxiliary diffusion from $J$. The original diffusion also remains in $J$. Thus, \cite[Corollary~2.2]{mijatovic2012martingale}, applied with integrand $-q$, shows that $\Lambda$ is a true martingale on $[0,T]$. Since $\Lambda_{0}=1$, $\mathbb{E}^{\mathbb{P}}[\Lambda_{t}]=1$ for all $0\leq t\leq T$. The final assertion follows from the strict positivity of $\Lambda_{T}$ and Girsanov's theorem.
\end{proof}

\subsection{CIR representation and local-to-CEV specialization}
\begin{theorem}[CKLS--CIR transform]\label{thm:general-ckls-cir}
Let $Y$ be the CKLS process with $a,b,\sigma,y_{0}>0$ and $k\in(1/2,1)$, and fix $T<\infty$ and $\zeta>0$. Under the equivalent measure $\mathbb{Q}$ constructed above, $\widetilde{W}$ is a standard Brownian motion, and $Z_{t}^{(k)}:=\mathcal{T}(Y_{t};k)$ by \eqref{state-mapping-general-form} satisfies
\begin{equation}\label{general-cir-transformed}
	dZ_{t}^{(k)}=\Big[\frac{\sigma^{2}\zeta^{2}}{2}-2b(1-k)Z_{t}^{(k)}\Big]dt+\sigma\zeta\sqrt{Z_{t}^{(k)}}d\widetilde{W}_{t},\qquad Z_{0}^{(k)}=\mathcal{T}(y_{0};k)>0.
\end{equation}

\end{theorem}
\begin{proof}
By Lemma~\ref{lem:ckls-density-martingale}, substituting $dW_{t}=d\widetilde{W}_{t}-q_{t}dt$ into \eqref{mapped-dynamics-general-P}, with $q_{t}$ from \eqref{girsanov-kernel-general}, yields \eqref{general-cir-transformed}.
\end{proof}
The CIR parameters are $\beta=2b(1-k)$, $\alpha=\sigma^{2}\zeta^{2}/2$, $\gamma=\sigma\zeta$, with $2\alpha=\gamma^{2}=\sigma^{2}\zeta^{2}$. For known $b$, the relation $k=1-\beta/(2b)$ therefore converts an estimator $\widehat{\beta}$ into $\widehat{k}=1-\widehat\beta/(2b)$. The free parameter $\zeta$ scales $\alpha$ and $Z_{0}$ quadratically and $\gamma$ linearly, without changing $\beta$ or the Girsanov kernel.

\begin{remark}[Parameter restrictions]\label{rem:ckls-parameter-restrictions}
By \cite{ning2025ckls}: for $0<k<1/2$, the Yamada--Watanabe criterion does not apply, and $0$ is attainable; the singularities of $\partial_{yy}\mathcal{T}$ and the Girsanov kernel in \eqref{girsanov-kernel-general} require localization and separate boundary analysis. The endpoint $k=1/2$ is already of CIR type, with $0$ attainable if $2a<\sigma^{2}$, and is omitted here. At $k=1$, the logarithmic-square mapping fails global square-root matching. For $k>1$, the diffusion coefficient is $-\sigma\zeta\sqrt{Z}$; replacing $W$ by $-W$ corrects the sign, but $\beta=2b(1-k)<0$ remains outside the mean-reverting CIR framework.
\end{remark}

We now apply Theorem~\ref{thm:general-ckls-cir} to the auxiliary model \eqref{Local-to-CEV-CKLS}, with $a=\varepsilon^{2}a^{*}$, $b=-\mu$, $\sigma=\varepsilon\nu$, and $k=\vartheta\in\Theta\Subset(1/2,1)$. The Girsanov kernel becomes dependent on $\varepsilon$:
\[
	q_{t}^{\varepsilon,\vartheta}=\frac{\varepsilon a^{*}}{\nu}\big(Y_{t}^{\varepsilon,\vartheta}\big)^{-\vartheta}-\frac{\varepsilon\nu}{2}\big(Y_{t}^{\varepsilon,\vartheta}\big)^{\vartheta-1}=\varepsilon\Big[\frac{a^{*}}{\nu}\big(Y_{t}^{\varepsilon,\vartheta}\big)^{-\vartheta}-\frac{\nu}{2}\big(Y_{t}^{\varepsilon,\vartheta}\big)^{\vartheta-1}\Big].
\]
Let $\Lambda^{\varepsilon,\vartheta}$ denote the stochastic exponential in \eqref{girsanov-kernel-general} with this kernel. For each fixed $\varepsilon>0$, it defines an equivalent probability measure $\mathbb{Q}^{\varepsilon,\vartheta}$ by $\left.d\mathbb{Q}^{\varepsilon,\vartheta}/d\mathbb{P}\right|_{\mathcal{F}_{T}}=\Lambda_{T}^{\varepsilon,\vartheta}$. Under $\mathbb{Q}^{\varepsilon,\vartheta}$, $\widetilde{W}_{t}^{\varepsilon,\vartheta}:=W_{t}+\int_{0}^{t}q_{s}^{\varepsilon,\vartheta}ds$ is a standard Brownian motion, and $Z_{t}^{\varepsilon,\vartheta}:=\mathcal{T}(Y_{t}^{\varepsilon,\vartheta};\vartheta)$, by \eqref{general-cir-transformed}, satisfies
\begin{equation}\label{small-noise-cir-dynamics}
	dZ_{t}^{\varepsilon,\vartheta}=\Big[\frac{\varepsilon^{2}\nu^{2}\zeta^{2}}{2}+2\mu(1-\vartheta)Z_{t}^{\varepsilon,\vartheta}\Big]dt+\varepsilon\nu\zeta \sqrt{Z_{t}^{\varepsilon,\vartheta}}d\widetilde{W}_{t}^{\varepsilon,\vartheta},\qquad Z_{0}^{\varepsilon,\vartheta}=\mathcal{T}(x_{0};\vartheta)>0.
\end{equation}
Thus, $\beta(\vartheta)=-2\mu(1-\vartheta)$ gives $\vartheta=1+\beta(\vartheta)/(2\mu)$. Transferring asymptotic results as $\varepsilon\to0$ additionally requires quantitative control of $\Lambda_{T}^{\varepsilon,\vartheta_{0}}$ in \eqref{girsanov-kernel-general}, which is established below. Henceforth, we fix $\vartheta=\vartheta_{0}$ and write $\mathbb{Q}^{\varepsilon}=\mathbb{Q}^{\varepsilon,\vartheta_{0}}$.

\subsection{Inverse moments, Girsanov energy, total variation and transfer of limits}
Although the Girsanov kernel carries an explicit factor $\varepsilon$, it involves negative powers of the state and may be singular near $0$. Moreover, the CKLS coefficients $a_{\varepsilon}=\varepsilon^{2}a^{*}$ and
$\sigma_{\varepsilon}=\varepsilon\nu$ vary with $\varepsilon$, so a
bound uniform in $\varepsilon$ is not immediate. We therefore establish
the following inverse-moment estimate for the local-to-CEV small-noise CKLS process.

\begin{lemma}[Uniform inverse-moment bound]\label{lem:uniform-inverse-moment-small-noise}
Under Assumptions~\ref{ass:A2} and~\ref{ass:A3}, for every $r>0$, $T\in(0,\infty)$, and $\varepsilon_{0}\in(0,1]$, there exists a constant $C_{r,T,\varepsilon_{0}}=C(r,T,\varepsilon_{0},a^{*},\mu,\nu,\vartheta_{0},x_{0})<\infty$ such that
\[
    \sup_{0<\varepsilon\leq\varepsilon_{0}}\sup_{0\leq t\leq T}\mathbb{E}^{\mathbb{P}}\big[\big(Y_{t}^{\varepsilon,\vartheta_{0}}\big)^{-r}\big]\leq C_{r,T,\varepsilon_{0}}.
\]
\end{lemma}
\begin{proof}
Let $f_{r}(y)=y^{-r}$ and stop $Y^{\varepsilon,\vartheta_{0}}$ at its exit time $\tau_{m}^{\varepsilon}$ from $(m^{-1},m)$, $m>\max\{x_{0},x_{0}^{-1}\}$. Its generator satisfies
\[
    \mathcal{L}^{\varepsilon}f_{r}(y)=-r\mu y^{-r}+\varepsilon^{2}\big[-ra^{*}y^{-r-1}+\frac{r(r+1)}{2}\nu^{2}y^{-r-2+2\vartheta_{0}}\big].
\]
With $x=y^{-r-1}$ and $\rho_{r}=(r+2-2\vartheta_{0})/(r+1)\in(0,1)$, the bracket equals $-ra^{*}x+\frac{r(r+1)}2\nu^{2}x^{\rho_{r}}\leq C_{r}$ for some $C_{r}>0$. Thus, for $0<\varepsilon\le\varepsilon_{0}$, $\mathcal{L}^{\varepsilon}f_{r}\leq r|\mu|f_{r}+\varepsilon_{0}^{2}C_{r}$. The stopped It\^{o} formula and Gr\"{o}nwall's inequality give
\[
    \mathbb{E}^{\mathbb{P}}\big[f_{r}\big(Y_{t\wedge\tau_{m}^{\varepsilon}}^{\varepsilon,\vartheta_{0}}\big)\big]\leq\big(x_{0}^{-r}+\varepsilon_{0}^{2}C_{r}T\big)e^{r|\mu|T},\qquad 0\leq t\leq T.
\]
For each $\varepsilon>0$, strict positivity and non-explosion imply $\tau_{m}^{\varepsilon}\uparrow\infty~a.s.$. Fatou's lemma yields the claimed result.
\end{proof}

For probability measures $Q$ and $P$ on $\mathcal{F}_{T}$, define the total variation $\|Q-P\|_{\mathrm{TV}}:=\sup_{A\in\mathcal{F}_{T}}|Q(A)-P(A)|$.
\begin{theorem}[Girsanov energy and asymptotic measure closeness]
\label{thm:total-variation-closeness}
Under the conditions of
Lemma~\ref{lem:uniform-inverse-moment-small-noise},
as $\varepsilon\to0$,
\[
    \mathbb{E}^{\mathbb{P}}\Big[\int_{0}^{T}\big|q_{t}^{\varepsilon,\vartheta_{0}}\big|^{2}dt\Big]=O(\varepsilon^{2}),\qquad\|\mathbb{Q}^{\varepsilon}-\mathbb{P}\|_{\mathrm{TV}}=\frac{1}{2}\mathbb{E}^{\mathbb{P}}\big[|\Lambda_{T}^{\varepsilon,\vartheta_{0}}-1|\big]=O(\varepsilon).
\]
Consequently, $\Lambda_{T}^{\varepsilon,\vartheta_{0}}\to1$ in $L^{1}(\mathbb{P})$ and $\log\Lambda_{T}^{\varepsilon,\vartheta_{0}}\to0$ in $\mathbb{P}$-probability.
\end{theorem}

\begin{proof}
With the inequality $|x+y|^{2}\leq 2|x|^{2}+2|y|^{2}$, the kernel representation and Lemma~\ref{lem:uniform-inverse-moment-small-noise} give
\[
    \mathbb{E}^{\mathbb{P}}\Big[\int_{0}^{T}\big|q_{t}^{\varepsilon,\vartheta_{0}}\big|^{2}dt\Big]\leq C\varepsilon^{2}\int_{0}^{T}\mathbb{E}^{\mathbb{P}}\Big[\big(Y_{t}^{\varepsilon,\vartheta_{0}}\big)^{-2\vartheta_{0}}+\big(Y_{t}^{\varepsilon,\vartheta_{0}}\big)^{-2(1-\vartheta_{0})}\Big]dt\leq C_{T}\varepsilon^{2}.
\]
Thus, $\int_{0}^{T}q_{t}^{\varepsilon,\vartheta_{0}}dW_{t}$ is square-integrable with mean $0$. Let $H(\mathbb{P}\mid\mathbb{Q}^{\varepsilon}):=\mathbb{E}^{\mathbb{P}}[-\log\Lambda_{T}^{\varepsilon,\vartheta_{0}}]$ denote the relative entropy (Kullback–Leibler divergence). The density formula in \eqref{girsanov-kernel-general} and Pinsker's inequality yield
\[
    \|\mathbb{Q}^{\varepsilon}-\mathbb{P}\|_{\mathrm{TV}}^{2}\leq\frac{1}{2} H(\mathbb{P}\mid\mathbb{Q}^{\varepsilon})=\frac{1}{4}\mathbb{E}^{\mathbb{P}}\Big[\int_{0}^{T}\big|q_{t}^{\varepsilon,\vartheta_{0}}\big|^{2}dt\Big]=O(\varepsilon^{2}).
\]
The stated $L^{1}$ bound follows from $\mathbb{E}^{\mathbb{P}}[|\Lambda_{T}^{\varepsilon,\vartheta_{0}}-1|]=2\|\mathbb{Q}^{\varepsilon}-\mathbb{P}\|_{\mathrm{TV}}$. Finally, convergence of $\log\Lambda_{T}^{\varepsilon,\vartheta_{0}}$ follows from convergence of $\Lambda_{T}^{\varepsilon,\vartheta_{0}}$ in probability and continuity of the logarithm at $1$.
\end{proof}

\begin{remark}[Role of the Local-to-CEV scaling]\label{rem:why-Local-to-CEV-scaling}
The additional drift $a_{\varepsilon}=\varepsilon^{2}a^{*}$ regularizes the $0$ boundary by keeping the auxiliary CKLS process strictly positive for each $\varepsilon>0$, while preserving the first-order small-noise behavior of the target CEV process; see \eqref{process-level-estiamte}. The common factor $\varepsilon$ in the Girsanov kernel and the uniform inverse-moment bound yield the $O(\varepsilon^{2})$ energy estimate in Theorem~\ref{thm:total-variation-closeness}. The same scaling supports the uniform inverse-moment estimate: in the generator applied to $y^{-r}$, the negative regularization term and the positive It\^{o} correction both carry a factor $\varepsilon^{2}$, while the former is more singular near $0$ because $\vartheta_{0}>1/2$. 

More generally, $a_{\varepsilon}=\varepsilon^{p}a^{*}$ with $p>1$ lies below the first-order small-noise scale and yields Girsanov kernel contributions of orders $\varepsilon^{p-1}$ and $\varepsilon$ on compact subsets of $(0,\infty)$. We choose $p=2$ to match these orders. By contrast, $a_{\varepsilon}=\varepsilon a^{*}$ produces an order-one contribution to the kernel on compact subsets of $(0,\infty)$ and enters the first-order fluctuation dynamics. Vanishing energy, and hence total-variation closeness, is therefore no longer guaranteed by the present argument. A fixed $a>0$ would change the deterministic small-noise limit, so the auxiliary-to-CEV approximation and common deterministic limit in Lemmas~\ref{lem:local-to-CEV-process-approx} and~\ref{lem:uniform-deterministic-limit} would no longer hold. It would also contribute a term of order $\varepsilon^{-1}$ to the kernel on such compact sets.
\end{remark}

\begin{theorem}[Transfer of convergence]\label{thm:transfer-convergence}
Under the setting of Theorem~\ref{thm:total-variation-closeness}, let
$(\mathsf{E},d_{\mathsf{E}})$ be a Polish space. For any $\mathsf{E}$-valued, $\mathcal{F}_{T}$-measurable random elements $U_{\varepsilon,n}$ and $U$,
\[
	U_{\varepsilon,n}\xrightarrow{\mathbb{Q}^{\varepsilon}\text{-prob.}}U\quad\Longleftrightarrow\quad U_{\varepsilon,n}\xrightarrow{\mathbb{P}\text{-prob.}}U.
\]
Moreover, weak convergence under $\mathbb{Q}^{\varepsilon}$ is equivalent to weak convergence to the same limiting law under $\mathbb{P}$.
\end{theorem}

\begin{proof}
For every $\eta>0$, define $A_{\varepsilon,n}(\eta):=\{d_{\mathsf{E}}(U_{\varepsilon,n},U)>\eta\}\in\mathcal{F}_{T}$, then
\[
	\big|\mathbb{Q}^{\varepsilon}(A_{\varepsilon,n}(\eta))-\mathbb{P}(A_{\varepsilon,n}(\eta))\big|\leq \|\mathbb{Q}^{\varepsilon}-\mathbb{P}\|_{\mathrm{TV}}\to0,\qquad\text{as}\quad\varepsilon\to0,
\]
which proves the probability statement. The weak-convergence statement follows similarly from the total-variation bound for expectations of bounded continuous functions.
\end{proof}

\section{Supporting Results for the Main Results on the Two-Step Estimator}\label{bbb}
This appendix develops the technical arguments supporting the main result for the two-step estimator $\widehat{\vartheta}_{\varepsilon,n}$. Subsection~\ref{subsec:B1} establishes the process-level local-to-CEV approximation and the common deterministic small-noise limit, providing the path-comparison and localization estimates needed below. Following the structure of Remark~\ref{rem:estimator-hierarchy}, the subsequent analysis connects the auxiliary oracle, the CEV-based oracle, and the feasible estimator. Specifically, Subsection~\ref{subsec:B2} derives the consistency and first-order limit of $\widehat{\beta}_{\varepsilon,n}^{Y,\circ}$ under $\mathbb{Q}^{\varepsilon}$ and transfers these results to $\mathbb{P}$ using Theorem~\ref{thm:transfer-convergence}. It also shows that replacing the auxiliary path by the target CEV path changes the oracle estimator by $o_{\mathbb{P}}(\varepsilon)$, yielding the same first-order limit for $\widehat{\beta}_{\varepsilon,n}^{X,\circ}$. Subsection~~\ref{subsec:B3} establishes the convergence rate of the preliminary estimator $\widetilde{\vartheta}_{\varepsilon,n}$. Finally, Subsection~~\ref{subsec:B4} combines this rate with a uniform parameter-sensitivity bound to show that replacing $\vartheta_{0}$ by $\widetilde{\vartheta}_{\varepsilon,n}$ in the CEV-based state mapping produces an $o_{\mathbb{P}}(\varepsilon)$ plug-in error. Together, these results justify the passage from the auxiliary oracle limit to the feasible estimator under the original CEV observation model.

\subsection{Local-to-CEV process approximation and deterministic limit}\label{subsec:B1}
We begin with two process-level results: a localized comparison of the auxiliary CKLS and target CEV processes, and their uniform convergence to a common deterministic limit.

The deterministic small-noise limit associated with \eqref{small-noise-CEV-def} solves $d\mathcal{X}_{t}=\mu\mathcal{X}_{t}dt$ with $\mathcal{X}_{0}=x_{0}$, and hence is given by $\mathcal{X}_{t}=x_{0}e^{\mu t}$. Since $\mu<0$,
\begin{equation}\label{deterministic-limit-bounds}
    0<x_{0}e^{\mu T}\leq\mathcal{X}_{t}\leq x_{0},\qquad 0\leq t\leq T.
\end{equation}
The following lemmas establish
\[
    \sup_{0\leq t\leq T}|X_t^{\varepsilon,\vartheta_{0}}-\mathcal{X}_t|=O_{\mathbb{P}}(\varepsilon),\qquad\sup_{0\leq t\leq T} |Y_{t}^{\varepsilon,\vartheta_{0}}-\mathcal{X}_t|=O_{\mathbb{P}}(\varepsilon).
\]
The deterministic bounds and uniform convergence ensure that both processes remain in a common compact subset of $(0,\infty)$ over $[0,T]$ with probability tending to $1$. Together with the localized path-comparison estimate, this yields
\[
    \sup_{0\leq t\leq T}\big|Y_{t}^{\varepsilon,\vartheta_{0}}-X_{t}^{\varepsilon,\vartheta_{0}}\big|=O_{\mathbb{P}}(\varepsilon^2)=o_{\mathbb{P}}(\varepsilon),
\]
formalizing the local-to-CEV interpretation.

\begin{lemma}[Process-level local-to-CEV approximation]\label{lem:local-to-CEV-process-approx}
Under Assumptions~\ref{ass:A2} and~\ref{ass:A3}, fix $T\in(0,\infty)$. For any $\vartheta\in\Theta$, suppose that $X^{\varepsilon,\vartheta}$ and $Y^{\varepsilon,\vartheta}$ are defined on the same filtered probability space, driven by the same Brownian motion $W$, and solve \eqref{small-noise-CEV-def} and \eqref{Local-to-CEV-CKLS}, respectively. For fixed $0<\delta<M<\infty$, define the compact localization event
\[
    \Omega_{\delta,M}^{\varepsilon,\vartheta}:=\Big\{\delta\leq\inf_{0\leq t\leq T}\min\{Y_{t}^{\varepsilon,\vartheta},X_{t}^{\varepsilon,\vartheta}\}\leq\sup_{0\leq t\leq T}\max\{Y_{t}^{\varepsilon,\vartheta},X_{t}^{\varepsilon,\vartheta}\}\leq M \Big\}.
\]
Then there exists a constant $C_{\delta,M,T}>0$, independent of $\varepsilon$, such that
\[
	\mathbb{E}\big[\sup_{0\leq t\leq T}\big|Y_{t}^{\varepsilon,\vartheta}-X_{t}^{\varepsilon,\vartheta}\big|^{2}\mathds{1}_{\Omega_{\delta,M}^{\varepsilon,\vartheta}}\big]\leq C_{\delta,M,T}\varepsilon^{4},\quad\text{and consequently, }\sup_{0\leq t\leq T}\big|Y_{t}^{\varepsilon,\vartheta}-X_{t}^{\varepsilon,\vartheta}\big|\mathds{1}_{\Omega_{\delta,M}^{\varepsilon,\vartheta}}=O_{\mathbb{P}}(\varepsilon^{2}).
\]
\end{lemma}
\begin{proof}
Set $D_{t}:=Y_{t}^{\varepsilon,\vartheta}-X_{t}^{\varepsilon,\vartheta}$ and
$\tau:=\inf\{t\geq0:(X_{t}^{\varepsilon,\vartheta},Y_{t}^{\varepsilon,\vartheta})
\notin(\delta/2,2M)^{2}\}$. On $[\delta/2,2M]$, the map $x\mapsto x^\vartheta$ is Lipschitz uniformly in $\vartheta\in\Theta$. Subtracting the two SDEs, stopping at $\tau$, and applying the Cauchy--Schwarz and BDG inequalities therefore give,
for $0\leq t\leq T$ and $0<\varepsilon\leq1$,
\[
    \mathbb{E}\Big[\sup_{0\leq r\leq t}|D_{r\wedge\tau}|^{2}\Big]\leq C\varepsilon^{4}+C\int_{0}^{t}\mathbb{E}\Big[\sup_{0\leq r\leq s}|D_{r\wedge\tau}|^{2}\Big]ds,
\]
where $C$ is independent of $\varepsilon$ and $\vartheta$. Since $\tau>T$ on $\Omega_{\delta,M}^{\varepsilon,\vartheta}$, Gr\"{o}nwall's inequality yields
\begin{equation}\label{L2-estimate-D}
    \mathbb{E}\Big[\sup_{0\leq t\leq T}|D_{t}|^{2}\mathds{1}_{\Omega_{\delta,M}^{\varepsilon,\vartheta}}\Big]\leq\mathbb{E}\Big[\sup_{0\leq t\leq T}|D_{t\wedge\tau}|^{2}\Big]\leq C_{\delta,M,T}\varepsilon^{4}.
\end{equation}
The stated $O_{\mathbb{P}}(\varepsilon^{2})$ bound follows from Markov's inequality.
\end{proof}

\begin{lemma}[Uniform deterministic small-noise limit]\label{lem:uniform-deterministic-limit}
Under Assumptions~\ref{ass:A1},~\ref{ass:A2},~\ref{ass:A3}, for any $\vartheta\in\Theta$,
\[
	\sup_{0\leq t\leq T}\big|X_{t}^{\varepsilon,\vartheta}-\mathcal{X}_{t}\big|=O_{\mathbb{P}}(\varepsilon),\qquad
	\sup_{0\leq t\leq T}\big|Y_{t}^{\varepsilon,\vartheta}-\mathcal{X}_{t}\big|=O_{\mathbb{P}}(\varepsilon).
\]
\end{lemma}
\begin{proof}
\textbf{Step~1 (Uniform moment bounds)}: For $m>x_{0}$, let $\tau_{m}:=\inf\{t\geq0:\max(X_{t}^{\varepsilon,\vartheta},Y_{t}^{\varepsilon,\vartheta})\geq m\}$. The stopped It\^{o} formulas have bounded stochastic integrands and hence mean-zero martingale terms. Using $x^{2\vartheta}\leq1+x^{2}$ and $2a^{*}x\leq(a^{*})^{2}+x^{2}$, we obtain, for $0<\varepsilon\leq1$ and $t\leq T$,
\[
    \mathbb{E}^{\mathbb{P}}\big[\big(X_{t\wedge\tau_{m}}^{\varepsilon,\vartheta}\big)^{2}+(Y_{t\wedge\tau_{m}}^{\varepsilon,\vartheta}\big)^{2}\big]\leq 2x_{0}^{2}+C\int_{0}^{t}\Big[1+\mathbb{E}^{\mathbb{P}}\big[\big(X_{s\wedge\tau_{m}}^{\varepsilon,\vartheta}\big)^{2}+\big(Y_{s\wedge\tau_{m}}^{\varepsilon,\vartheta}\big)^{2}\big]\Big]ds,
\]
where $C$ is independent of $m$, $\varepsilon$, $\vartheta$. Gr\"{o}nwall's inequality and Fatou's lemma, using nonexplosion to let $m\to\infty$, yield
\begin{equation}\label{X-Y-uniform-second-moment}
    \sup_{0<\varepsilon\leq1}\sup_{\vartheta\in\Theta}\sup_{0\leq t\leq T}\mathbb{E}^{\mathbb{P}}\big[\big(X_{t}^{\varepsilon,\vartheta}\big)^{2}+\big(Y_{t}^{\varepsilon,\vartheta}\big)^{2}\big]\leq C_{T}<\infty.
\end{equation}

\textbf{Step~2 (Approximation by the deterministic path)}: Variation of constants gives
\[
\begin{aligned}
    X_{t}^{\varepsilon,\vartheta}-\mathcal{X}_{t}=\varepsilon\nu e^{\mu t}\int_{0}^{t}e^{-\mu s}\big(X_{s}^{\varepsilon,\vartheta}\big)^{\vartheta}dW_{s},\qquad
    Y_{t}^{\varepsilon,\vartheta}-\mathcal{X}_{t}=\varepsilon^{2}a^{*}e^{\mu t}\int_{0}^{t}e^{-\mu s}ds+\varepsilon\nu e^{\mu t}\int_{0}^{t}e^{-\mu s}\big(Y_{s}^{\varepsilon,\vartheta}\big)^{\vartheta}dW_{s}.
\end{aligned}
\]
Since $e^{\pm\mu t}$ are bounded on $[0,T]$, the BDG inequality, $x^{2\vartheta}\leq1+x^{2}$, and \eqref{X-Y-uniform-second-moment} imply
\[
    \mathbb{E}^{\mathbb{P}}\Big[\sup_{0\leq t\leq T}\big(\big|X_{t}^{\varepsilon,\vartheta}-\mathcal{X}_{t}\big|^{2}+\big|Y_{t}^{\varepsilon,\vartheta}-\mathcal{X}_{t}\big|^{2}\big)\Big]\leq C_{T}\varepsilon^{4}+C_{T}\varepsilon^{2}\int_{0}^{T}\Big[1+\mathbb{E}^{\mathbb{P}}\big[\big(X_{s}^{\varepsilon,\vartheta}\big)^{2}+\big(Y_{s}^{\varepsilon,\vartheta}\big)^{2}\big]\Big]ds\leq C_{T}\varepsilon^{2},
\]
uniformly in $\vartheta\in\Theta$. Both assertions follow from Markov's inequality.
\end{proof}

\subsection{Oracle LSE, asymptotics, and auxiliary-to-CEV equivalence}\label{subsec:B2}
We first analyze the auxiliary oracle estimator $\widehat{\beta}_{\varepsilon,n}^{Y,\circ}$ defined in Subsection~\ref{subsec:feasible-oracles}. Recall that, under $\mathbb{Q}^{\varepsilon}$, the transformed auxiliary process $Z_{t}^{Y,\varepsilon,\circ}:=\mathcal{T}(Y_{t}^{\varepsilon,\vartheta_{0}};\vartheta_{0})$ satisfies the pseudo-CIR dynamics
\begin{equation}\label{auxiliary-oracle-cir-dynamics}
    dZ_{t}^{Y,\varepsilon,\circ}=\Big(\frac{\varepsilon^{2}\nu^{2}\zeta^{2}}{2}-\beta_{0}Z_{t}^{Y,\varepsilon,\circ}\Big)dt+\varepsilon\nu\zeta\sqrt{Z_{t}^{Y,\varepsilon,\circ}}d\widetilde{W}_{t}^{\varepsilon},\qquad Z_{0}^{Y,\varepsilon,\circ}=\frac{\zeta^{2}}{4(1-\vartheta_{0})^{2}}x_{0}^{2(1-\vartheta_{0})}.
\end{equation}

\begin{remark}[Relation to S{\o}rensen and Uchida's framework]
The oracle discretization argument requires the additional relation $\Delta_{n}=o(\varepsilon)$, corresponding to the small-noise drift-estimation condition of \cite{sorensen2003small} when the diffusion coefficient is known. This condition is weaker than Assumption~\ref{ass:A5}, which implies $\Delta_{n}=o(\varepsilon^{3})$. The framework of \cite{sorensen2003small} imposes smoothness and uniform nondegeneracy conditions on the diffusion coefficients.
\end{remark}
For the auxiliary CIR dynamics, the Feller condition $2\cdot\varepsilon^{2}\nu^{2}\zeta^{2}/2=\varepsilon^{2}\nu^{2}\zeta^{2}$ ensures strict positivity of $Z_{t}^{Y,\varepsilon,\circ}$, providing a strictly positive CIR benchmark for localized asymptotic analysis. This allows the establishment of the oracle asymptotics without directly confronting the boundary degeneracy associated with the target CEV process hitting and being absorbed at $0$. Meanwhile, the square-root diffusion coefficient has derivatives that are singular at $0$ and its square is not bounded away from $0$ on $(0,\infty)$. We therefore verify the required conditions on a compact neighborhood of the deterministic limiting path contained in $(0,\infty)$, and show below that the auxiliary process remains in this neighborhood with probability tending to $1$.

Using the state mapping \eqref{state-mapping-general-form} for $x>0$, define the transformed deterministic path
\[
    \mathcal{Z}_{t}:=\mathcal{T}(\mathcal{X}_{t};\vartheta_{0})=\frac{\zeta^{2}}{4(1-\vartheta_{0})^{2}}
\mathcal{X}_{t}^{2(1-\vartheta_{0})}.
\]
Notably, $\mathcal{Z}$ is the image of the deterministic small-noise limit $\mathcal{X}$ under the state mapping. This definition does not assert that the transformed absorbing CEV process satisfies the CIR dynamics globally after hitting $0$.  Equivalently, it coincides with the deterministic limit of the corresponding pseudo-CIR dynamics. Since $\mathcal{X}_{t}=x_{0}e^{\mu t}$ and $\beta_{0}=-2\mu(1-\vartheta_{0})$, we have $\mathcal{Z}_{t}=\mathcal{Z}_{0}e^{-\beta_{0}t}$. Hence, by \eqref{deterministic-limit-bounds}, $0<\inf_{0\leq t\leq T}\mathcal{Z}_{t}\leq\sup_{0\leq t\leq T}\mathcal{Z}_{t}<\infty$. Together with Lemma~\ref{lem:uniform-deterministic-limit} and Theorem~\ref{thm:transfer-convergence}, this ensures that the transformed auxiliary process remains, with probability tending to $1$, in a neighborhood of $\mathcal{Z}$ bounded away from $0$. This localization provides the regularity needed for the
oracle analysis in the spirit of \cite{sorensen2003small}.
Theorem~\ref{thm:auxiliary-oracle-lse} below establishes the
oracle limit directly for the present $\varepsilon$-dependent
auxiliary dynamics.

Before developing the oracle asymptotics of Theorem~\ref{thm:auxiliary-oracle-lse}, we first establish the following lemma.
\begin{lemma}[Moment and increment bounds for $Z_{u}^{Y,\varepsilon,\circ}$]\label{lem:auxiliary-pseudo-cir-bounds}
Let $Z^{Y,\varepsilon,\circ}$ be the auxiliary pseudo-CIR process satisfying
\[
    \text{\eqref{auxiliary-oracle-cir-dynamics}. Then}\sup_{0<\varepsilon\le\varepsilon_{0}}\sup_{0\leq u\leq T}\mathbb{E}^{\mathbb{Q}^{\varepsilon}}\big[Z_{u}^{Y,\varepsilon,\circ}+\big(Z_{u}^{Y,\varepsilon,\circ}\big)^{2}+\big(Z_{u}^{Y,\varepsilon,\circ}\big)^{3}\big]<\infty.\text{ Moreover, there exists some $C>0$ such that}
\]
\begin{equation}\label{auxiliary-oracle-increment-estimate-general}
    \mathbb{E}^{\mathbb{Q}^{\varepsilon}}\big[|Z_{u}^{Y,\varepsilon,\circ}-Z_{s}^{Y,\varepsilon,\circ}|^{2}\big]\leq C\big[|u-s|^{2}+\varepsilon^{2}|u-s|\big],\qquad 0\leq s\leq u\leq T.
\end{equation}
\end{lemma}
\begin{proof}
Write $Z_{t}=Z_{t}^{Y,\varepsilon,\circ}$ and $z_{0}=Z_{0}$. For $m>z_{0}$, let
$\tau_{m}=\inf\{t\geq0:Z_{t}\geq m\}$. Applying It\^{o}'s formula to $(Z_{t\wedge\tau_{m}})^{p}$, $p=1,2,3$, and taking expectations, the stopped martingale
term has mean $0$. Dropping the nonpositive drift term and using $\alpha_{\varepsilon}=O(\varepsilon^{2})$ gives
\[
    \mathbb{E}^{\mathbb{Q}^{\varepsilon}}\Big[\big(Z_{t\wedge\tau_{m}}\big)^{p}\Big]\leq z_{0}^{p}+C\varepsilon^{2}\int_{0}^{t}\mathbb{E}^{\mathbb{Q}^{\varepsilon}}\Big[\big(Z_{v\wedge\tau_{m}}\big)^{p-1}\Big]dv,\qquad p=1,2,3,
\]
where $Z^{0}=1$ and $C$ is independent of $\varepsilon$ and $m$. Induction over $p$ yields moment bounds uniform in $t\in[0,T]$, $m$, and $0<\varepsilon\leq\varepsilon_{0}$. Nonexplosion and Fatou's lemma then give the asserted bounds for $Z_{t}$.

For $0\leq s\leq u\leq T$, the SDE, the Cauchy--Schwarz inequality, and It\^{o}'s isometry yield
\[
    \mathbb{E}^{\mathbb{Q}^{\varepsilon}}\big[|Z_u-Z_s|^{2}\big]\leq 2(u-s)\int_{s}^{u}\mathbb{E}^{\mathbb{Q}^{\varepsilon}}\big[|\alpha_{\varepsilon}-\beta_{0}Z_{v}|^{2}\big]dv+2\varepsilon^{2}\gamma^{2}\int_{s}^{u}\mathbb{E}^{\mathbb{Q}^{\varepsilon}}[Z_{v}]dv\leq C\big((u-s)^{2}+\varepsilon^{2}(u-s)\big),
\]
by the uniform first- and second-moment bounds.
\end{proof}

\begin{theorem}[Oracle LSE under the Auxiliary Measure]
\label{thm:auxiliary-oracle-lse}
Suppose that Assumptions~\ref{ass:A1},~\ref{ass:A2},~\ref{ass:A3} hold, together with $\Delta_{n}=o(\varepsilon)$. Then the auxiliary oracle estimator $\widehat{\beta}^{Y,\circ}_{\varepsilon,n}$ satisfies
\[
      \widehat{\beta}_{\varepsilon,n}^{Y,\circ}\xrightarrow{\mathbb{Q}^{\varepsilon}\text{-prob.}}\beta_{0}\quad\text{and}\quad\varepsilon^{-1}\big(\widehat{\beta}_{\varepsilon,n}^{Y,\circ}-\beta_{0}\big)\xrightarrow{\mathbb{Q}^{\varepsilon}\text{-prob.}}\xi_{\beta_{0}},\quad\text{where}\quad\xi_{\beta_{0}}:=-\gamma\frac{\int_{0}^{T}(\mathcal{Z}_{t})^{3/2}dW_{t}}{\int_{0}^{T}(\mathcal{Z}_{t})^{2}dt}.
\]
Moreover, $\xi_{\beta_{0}}$ is $\mathcal{F}_{T}$-measurable and $\mathcal{L}_{\mathbb{P}}(\xi_{\beta_{0}})=\mathcal{N}(0,V_{\beta_{0}})$, where $V_{\beta_{0}}=\big(\gamma^{2}\int_{0}^{T}(\mathcal{Z}_{t})^{3}dt\big)\big/\big(\int_{0}^{T}(\mathcal{Z}_{t})^{2}dt\big)^{2}$.
\end{theorem}
\begin{proof}
Write $Z_{t}=Z_{t}^{Y,\varepsilon,\circ}$ and let $\eta_{n}(s)=t_{i-1}$ for $s\in(t_{i-1},t_{i}]$. 

\textbf{Step~1 (Error decomposition)}: For $\beta\in\mathbb{R}$, define the unweighted Euler contrast
\[
    \Psi_{\varepsilon,n}^{Y}(\beta):=\frac{1}{\varepsilon^{2}\Delta_{n}}\sum_{i=1}^{n}\big[\Delta_{i}Z-(\alpha_{\varepsilon}-\beta Z_{t_{i-1}})\Delta_{n}\big]^{2}.
\]
Differentiation gives
\[
    \partial_{\beta}\Psi_{\varepsilon,n}^{Y}(\beta)=\frac{2}{\varepsilon^{2}}\sum_{i=1}^{n}Z_{t_{i-1}}\big[\Delta_{i}Z-\alpha_{\varepsilon}\Delta_{n}+\beta Z_{t_{i-1}}\Delta_{n}\big],\qquad\partial_{\beta}^{2}\Psi_{\varepsilon,n}^{Y}(\beta)=\frac{2\Delta_{n}}{\varepsilon^{2}}\sum_{i=1}^{n}\big(Z_{t_{i-1}}\big)^{2}>0,
\]
where strict positivity follows from $Z_{0}>0$. Thus, solving $\partial_{\beta}\Psi_{\varepsilon,n}^{Y}(\beta)=0$ gives the unique minimizer $\widehat\beta_{\varepsilon,n}^{Y,\circ}=\mathcal{B}_{\varepsilon,n}(Z)$ defined by \eqref{path-functional-B}. Define
\[
    D_{\varepsilon,n}^{Y}:=\int_{0}^{T}\big(Z_{\eta_{n}(s)}\big)^{2}ds,\qquad
    M_{\varepsilon,n}^{Y}:=\int_{0}^{T}Z_{\eta_{n}(s)}\sqrt{Z_{s}}d\widetilde{W}_{s}^{\varepsilon},\qquad 
    R_{\varepsilon,n}^{Y}:=-\beta_{0}\int_{0}^{T}Z_{\eta_{n}(s)}(Z_{s}-Z_{\eta_{n}(s)})ds.
\]
Substituting \eqref{auxiliary-oracle-cir-dynamics} into \eqref{path-functional-B} gives
\begin{equation}\label{contrast-derived-normalized-error}
    \varepsilon^{-1}(\widehat\beta_{\varepsilon,n}^{Y,\circ}-\beta_{0})=-\gamma\frac{M_{\varepsilon,n}^{Y}}{D_{\varepsilon,n}^{Y}}-\frac{R_{\varepsilon,n}^{Y}}{\varepsilon D_{\varepsilon,n}^{Y}},
\end{equation}
where $D_{\varepsilon,n}^{Y}\geq\Delta_{n}(Z_{0})^{2}>0$.

\textbf{Step~2 (Denominator and remainder)}: Lemma~\ref{lem:uniform-deterministic-limit}, local Lipschitz continuity of $\mathcal{T}(\cdot;\vartheta_{0})$ near the positive trajectory $\mathcal{X}$, and
Theorem~\ref{thm:transfer-convergence} yield $\sup_{t\leq T}|Z_{t}-\mathcal{Z}_{t}|=o_{\mathbb{Q}^{\varepsilon}}(1)$. Since $\mathcal{Z}$ is continuous and strictly positive,
\begin{equation}\label{contrast-denominator-limit}
    D_{\varepsilon,n}^{Y}\xrightarrow{\mathbb{Q}^{\varepsilon}\text{-prob.}}D_{0}:=\int_{0}^{T}\big(\mathcal{Z}_{t}\big)^{2}dt>0.
\end{equation}
By the Cauchy--Schwarz inequality and the moment and increment bounds in Lemma~\ref{lem:auxiliary-pseudo-cir-bounds},
\[
    \mathbb{E}^{\mathbb{Q}^{\varepsilon}}\big[|R_{\varepsilon,n}^{Y}|\big]\leq C\int_{0}^{T}\big[(s-\eta_{n}(s))^{2}+\varepsilon^{2}(s-\eta_{n}(s))\big]^{1/2}ds\leq C(\Delta_{n}+\varepsilon\sqrt{\Delta_{n}}).
\]
Consequently, $\Delta_{n}=o(\varepsilon)$ implies $\varepsilon^{-1}R_{\varepsilon,n}^{Y}=O_{\mathbb{Q}^{\varepsilon}}(\Delta_{n}/\varepsilon+\sqrt{\Delta_{n}})=o_{\mathbb{Q}^{\varepsilon}}(1)$.

\textbf{Step~3 (Martingale approximation)}: Set $H_{\varepsilon,n}(s):=Z_{\eta_{n}(s)}\sqrt{Z_{s}}-(\mathcal{Z}_{s})^{3/2}$. The uniform convergence in \textbf{Step~2} and continuity of $\mathcal{Z}$ imply $\sup_{s\leq T}|H_{\varepsilon,n}(s)|=o_{\mathbb{Q}^{\varepsilon}}(1)$. Thus, the continuous local martingale
$N_{\varepsilon,n}(t):=\int_{0}^{t}H_{\varepsilon,n}(s)d\widetilde{W}_{s}^{\varepsilon}$ satisfies $\langle N_{\varepsilon,n}\rangle_{T}=\int_{0}^{T}|H_{\varepsilon,n}(s)|^{2}ds=o_{\mathbb{Q}^{\varepsilon}}(1)$. Lenglart's inequality therefore gives
\begin{equation}\label{oracle-martingale-tildeW}
    M_{\varepsilon,n}^{Y}=\int_{0}^{T}\big(\mathcal{Z}_{s}\big)^{3/2}d\widetilde{W}_{s}^{\varepsilon}+o_{\mathbb{Q}^{\varepsilon}}(1).
\end{equation}

\textbf{Step~4 (Measure transfer and conclusion)}: Since $d\widetilde{W}_{s}^{\varepsilon}=dW_{s}+q_{s}^{\varepsilon}ds$,
\[
    \int_{0}^{T}\big(\mathcal{Z}_{s}\big)^{3/2}d\widetilde{W}_{s}^{\varepsilon}=\int_{0}^{T}\big(\mathcal{Z}_{s}\big)^{3/2}dW_{s}+\int_{0}^{T}\big(\mathcal{Z}_{s}\big)^{3/2}q_{s}^{\varepsilon}ds.
\]
Theorem~\ref{thm:total-variation-closeness} and the Cauchy--Schwarz inequality give
\[
    \mathbb{E}^{\mathbb{P}}\Big[\Big|\int_{0}^{T}(\mathcal{Z}_{s})^{3/2}q_{s}^{\varepsilon}ds\Big|^{2}\Big]\leq\Big[\int_{0}^{T}\big(\mathcal{Z}_{s}\big)^{3}ds\Big]\mathbb{E}^{\mathbb{P}}\Big[\int_{0}^{T}\big|q_{s}^{\varepsilon}\big|^{2}ds\Big]=O(\varepsilon^{2}).
\]
The correction is therefore $o_{\mathbb{P}}(1)$ and, by Theorem~\ref{thm:transfer-convergence}, $o_{\mathbb{Q}^{\varepsilon}}(1)$. Hence,
\[
    M_{\varepsilon,n}^{Y}=\int_{0}^{T}\big(\mathcal{Z}_{s}\big)^{3/2}dW_{s}+o_{\mathbb{Q}^{\varepsilon}}(1).
\]
Combining the preceding estimates with \eqref{contrast-derived-normalized-error} yields
\[
    \varepsilon^{-1}(\widehat\beta_{\varepsilon,n}^{Y,\circ}-\beta_{0})\xrightarrow{\mathbb{Q}^{\varepsilon}\text{-prob.}}-\gamma\frac{\int_{0}^{T}(\mathcal{Z}_{s})^{3/2}dW_{s}}{D_{0}}=\xi_{\beta_{0}}.
\]
Moreover, \eqref{oracle-martingale-tildeW} gives $M_{\varepsilon,n}^{Y}=O_{\mathbb{Q}^{\varepsilon}}(1)$, so \eqref{contrast-derived-normalized-error} also yields consistency. Since $\mathcal{Z}$ is deterministic, $\xi_{\beta_{0}}$ is $\mathcal{F}_{T}$-measurable and has the stated Gaussian law under $\mathbb{P}$.
\end{proof}

Combining Theorem~\ref{thm:auxiliary-oracle-lse} and Theorem~\ref{thm:transfer-convergence} gives the corresponding result
under the original measure:
\begin{equation}\label{auxiliary-oracle-under-original-measure}
    \widehat{\beta}_{\varepsilon,n}^{Y,\circ}\xrightarrow{\mathbb{P}\text{-prob.}}\beta_{0},\qquad\varepsilon^{-1}\big(\widehat{\beta}_{\varepsilon,n}^{Y,\circ}-\beta_{0}\big)\xrightarrow{\mathbb{P}\text{-prob.}}\xi_{\beta_{0}},
\end{equation}
and $\mathcal{L}_{\mathbb{P}}(\xi_{\beta_{0}})=\mathcal{N}(0,V_{\beta_{0}})$, where $\xi_{\beta_{0}}$ and $V_{\beta_{0}}$ are as defined in Theorem~\ref{thm:auxiliary-oracle-lse}. In particular, the measure transfer does not alter the limiting random variable or its variance.

We next establish first-order equivalence between the auxiliary oracle and the CEV-based oracle estimator $\widehat{\beta}_{\varepsilon,n}^{X,\circ}$ defined in Subsection~\ref{subsec:feasible-oracles}.

\begin{theorem}[Auxiliary-to-CEV oracle equivalence]\label{thm:auxiliary-cev-oracle-equivalence}
Under the assumptions of Theorem~\ref{thm:auxiliary-oracle-lse},
\[
    \varepsilon^{-1}\big(\widehat{\beta}_{\varepsilon,n}^{X,\circ}-\widehat{\beta}_{\varepsilon,n}^{Y,\circ}\big)\xrightarrow{\mathbb{P}\text{-prob.}}0.
\]
\end{theorem}
\begin{proof}
\textbf{Step~1 (Localization and denominator comparison)}: Choose $0<\delta<\inf_{t\leq T}\mathcal{X}_{t}$ and $M>\sup_{t\leq T}\mathcal{X}_{t}$, define
\[
	\tau^{\varepsilon}:=\inf\big\{t\geq0:\big(X_{t}^{\varepsilon,\vartheta_{0}},Y_{t}^{\varepsilon,\vartheta_{0}}\big)\notin(\delta,M)^{2}\big\},\qquad\delta Z_{t}^{\varepsilon}:=Z_{t}^{X,\varepsilon,\circ}-Z_{t}^{Y,\varepsilon,\circ}.
\]
Lemma~\ref{lem:uniform-deterministic-limit} gives $\mathbb{P}(\tau^{\varepsilon}\leq T)\to0$. The stopped $L^{2}$ argument in Lemma~\ref{lem:local-to-CEV-process-approx}, together with the local Lipschitz continuity of $\mathcal{T}(\cdot;\vartheta_{0})$, yields
\begin{equation}\label{transformed-xy-path-difference}
	\mathbb{E}^{\mathbb{P}}\Big[\sup_{t\leq T}|\delta Z_{t\wedge\tau^{\varepsilon}}^{\varepsilon}|^{2}\Big]\leq C\varepsilon^{4},\qquad\sup_{t\leq T}|\delta Z_{t}^{\varepsilon}|=O_{\mathbb{P}}(\varepsilon^{2}).
\end{equation}
The second assertion follows from Markov's inequality and $\mathbb{P}(\tau^{\varepsilon}\leq T)\to0$, since the stopped and unstopped paths coincide on $\{\tau^{\varepsilon}>T\}$. Both transformed paths are also bounded in probability uniformly on $[0,T]$.

For a discrete path $V=\{V_{t_{i}}\}_{i=0}^{n}$, write
\[
	\mathfrak{N}_{\varepsilon,n}(V):=\alpha_{\varepsilon}\Delta_{n}\sum_{i=1}^{n}V_{t_{i-1}}-\sum_{i=1}^{n}V_{t_{i-1}}\Delta_{i}V,\qquad\mathfrak{D}_{n}(V):=\Delta_{n}\sum_{i=1}^{n}(V_{t_{i-1}})^{2}.
\]
The two oracle estimators are the corresponding ratios $\mathfrak{N}_{\varepsilon,n}(V)/\mathfrak{D}_{n}(V)$. Since $n\Delta_{n}=T$, \eqref{transformed-xy-path-difference} implies
\[
	\mathfrak{D}_{n}(Z^{X,\varepsilon,\circ})-\mathfrak{D}_{n}(Z^{Y,\varepsilon,\circ})=O_{\mathbb{P}}(\varepsilon^{2}).
\]
Moreover, \eqref{contrast-denominator-limit} and Theorem~\ref{thm:transfer-convergence} give $\mathfrak{D}_{n}(Z^{Y,\varepsilon,\circ})\xrightarrow{\mathbb{P}\text{-prob.}}D_{0}:=\int_{0}^{T}(\mathcal{Z}_{t})^{2}dt>0$. Thus, $\mathfrak{D}_{n}(Z^{X,\varepsilon,\circ})$ has the same limit, and its reciprocal is $O_{\mathbb{P}}(1)$.

\textbf{Step~2 (Squared-increment comparison)}: Under $\mathbb{P}$, It\^{o}'s formula before $\tau^{\varepsilon}$ gives $d\delta Z_{t}^{\varepsilon}=m_{t}^{\varepsilon}dt+\Sigma_{t}^{\varepsilon}dW_{t}$, where
\[
	m_{t}^{\varepsilon}:=-\beta_{0}\delta Z_{t}^{\varepsilon}-\varepsilon^{2}a^{*}\frac{\zeta^{2}}{2(1-\vartheta_{0})}\big(Y_{t}^{\varepsilon,\vartheta_{0}}\big)^{1-2\vartheta_{0}},\qquad\Sigma_{t}^{\varepsilon}:=\varepsilon\gamma\Big(\sqrt{Z_{t}^{X,\varepsilon,\circ}}-\sqrt{Z_{t}^{Y,\varepsilon,\circ}}\Big).
\]
On $[0,T\wedge\tau^{\varepsilon}]$, $\big(Y_{t}^{\varepsilon,\vartheta_{0}}\big)^{1-2\vartheta_{0}}$ is bounded, so $|m_{t}^{\varepsilon}|^{2}\leq C\big(|\delta Z_{t}^{\varepsilon}|^{2}+\varepsilon^{4}\big)$. Moreover, $|\sqrt{x}-\sqrt{y}|^{2}\leq|x-y|$ gives $|\Sigma_{t}^{\varepsilon}|^{2}\leq\varepsilon^{2}\gamma^{2}|\delta Z_{t}^{\varepsilon}|$. Thus, \eqref{transformed-xy-path-difference} and the Cauchy--Schwarz inequality yield
\[
\begin{aligned}
    \mathbb{E}^{\mathbb{P}}\Big[\int_{0}^{T\wedge\tau^{\varepsilon}}\big|m_{s}^{\varepsilon}\big|^{2}ds\Big]&\leq C\mathbb{E}^{\mathbb{P}}\Big[\sup_{t\leq T}|\delta Z_{t\wedge\tau^{\varepsilon}}^{\varepsilon}|^{2}\Big]+C\varepsilon^{4}\leq C\varepsilon^{4},\\
    \mathbb{E}^{\mathbb{P}}\Big[\int_{0}^{T\wedge\tau^{\varepsilon}}\big|\Sigma_{s}^{\varepsilon}\big|^{2}ds\Big]&\leq C\varepsilon^{2}\Big\{\mathbb{E}^{\mathbb{P}}\Big[\sup_{t\leq T}|\delta Z_{t\wedge\tau^{\varepsilon}}^{\varepsilon}|^{2}\Big]\Big\}^{1/2}\leq C\varepsilon^{4}.
\end{aligned}
\]
The Cauchy--Schwarz inequality and It\^{o}'s isometry therefore yield
\[
	\mathbb{E}^{\mathbb{P}}\Big[\sum_{i=1}^{n}\big|\delta Z_{t_{i}\wedge\tau^{\varepsilon}}^{\varepsilon}-\delta Z_{t_{i-1}\wedge\tau^{\varepsilon}}^{\varepsilon}\big|^{2}\Big]\leq 2\Delta_{n}\mathbb{E}^{\mathbb{P}}\Big[\int_{0}^{T\wedge\tau^{\varepsilon}}|m_{s}^{\varepsilon}|^{2}ds\Big]+2\mathbb{E}^{\mathbb{P}}\Big[\int_{0}^{T\wedge\tau^{\varepsilon}}|\Sigma_{s}^{\varepsilon}|^{2}ds\Big]\leq C\varepsilon^{4}.
\]
By Markov's inequality and $\mathbb{P}(\tau^{\varepsilon}\leq T)\to0$,
\begin{equation}\label{realized-qv-delta-Z-bound}
	\sum_{i=1}^{n}\big(\Delta_{i}\delta Z^{\varepsilon}\big)^{2}=O_{\mathbb{P}}(\varepsilon^{4}).
\end{equation}
Similarly, the SDEs, the Cauchy--Schwarz inequality, It\^{o}'s isometry, and \eqref{X-Y-uniform-second-moment}, using $x^{2\vartheta_{0}}\leq1+x^{2}$, give
\[
	\mathbb{E}^{\mathbb{P}}\Big[\sum_{i=1}^{n}\big(|\Delta_{i}X^{\varepsilon,\vartheta_{0}}|^{2}+|\Delta_{i}Y^{\varepsilon,\vartheta_{0}}|^{2}\big)\Big]\leq C(\varepsilon^{2}+\Delta_{n}).
\]
On $\{\tau^{\varepsilon}>T\}$, the mapping $\mathcal{T}(\cdot;\vartheta_{0})$ is Lipschitz on $[\delta,M]$. Hence, Markov's inequality gives
\begin{equation}\label{transformed-process-realized-qv-bound}
	\sum_{i=1}^{n}\Big[\big(\Delta_{i}Z^{X,\varepsilon,\circ}\big)^{2}+\big(\Delta_{i}Z^{Y,\varepsilon,\circ}\big)^{2}\Big]=O_{\mathbb{P}}(\varepsilon^{2}+\Delta_{n}).
\end{equation}
Combining \eqref{realized-qv-delta-Z-bound} and \eqref{transformed-process-realized-qv-bound} with the Cauchy--Schwarz inequality yields \begin{equation}\label{xy-realized-qv-difference}
	\Big|\sum_{i=1}^{n}\Big[\big(\Delta_{i}Z^{X,\varepsilon,\circ}\big)^{2}-\big(\Delta_{i}Z^{Y,\varepsilon,\circ}\big)^{2}\Big]\Big|\leq\Big[\sum_{i=1}^{n}\big(\Delta_{i}\delta Z^{\varepsilon}\big)^{2}\Big]^{\frac{1}{2}}\Big[\sum_{i=1}^{n}\big(\Delta_{i}Z^{X,\varepsilon,\circ}+\Delta_{i}Z^{Y,\varepsilon,\circ}\big)^{2}\Big]^{\frac{1}{2}}=O_{\mathbb{P}}\big(\varepsilon^{2}\sqrt{\varepsilon^{2}+\Delta_{n}}\big).
\end{equation}
\textbf{Step~3 (Numerator comparison and oracle equivalence)}: For any discrete path $V$,
\begin{equation}\label{discrete-integration-by-parts-square}
	\sum_{i=1}^{n}V_{t_{i-1}}\Delta_{i}V=\frac{1}{2}\Big[V_{T}^{2}-V_{0}^{2}-\sum_{i=1}^{n}(\Delta_{i}V)^{2}\Big].
\end{equation}
Since $Z_{0}^{X,\varepsilon,\circ}
=Z_{0}^{Y,\varepsilon,\circ}$,
\eqref{discrete-integration-by-parts-square} gives
\[
    \mathfrak{N}_{\varepsilon,n}(Z^{X,\varepsilon,\circ})-\mathfrak{N}_{\varepsilon,n}(Z^{Y,\varepsilon,\circ})=\alpha_{\varepsilon}\Delta_{n}\sum_{i=1}^{n}\delta Z_{t_{i-1}}^{\varepsilon}-\frac{1}{2}\Big[\big(Z_{T}^{X,\varepsilon,\circ}\big)^{2}-\big(Z_{T}^{Y,\varepsilon,\circ}\big)^{2}\Big]+\frac{1}{2}\sum_{i=1}^{n}\Big[\big(\Delta_iZ^{X,\varepsilon,\circ}\big)^{2}-\big(\Delta_iZ^{Y,\varepsilon,\circ}\big)^{2}\Big].
\]
The drift-sum term is $O_{\mathbb{P}}(\varepsilon^{4})$ by $\alpha_{\varepsilon}=O(\varepsilon^{2})$, $n\Delta_{n}=T$, and \eqref{transformed-xy-path-difference}. The endpoint difference equals $\delta Z_{T}^{\varepsilon}(Z_{T}^{X,\varepsilon,\circ}+Z_{T}^{Y,\varepsilon,\circ})$ and is $O_{\mathbb{P}}(\varepsilon^{2})$ by the same path bound and the stochastic boundedness of the transformed processes. Together with \eqref{xy-realized-qv-difference}, this yields
\begin{equation}\label{numerator-convergence}
    \mathfrak{N}_{\varepsilon,n}(Z^{X,\varepsilon,\circ})-\mathfrak{N}_{\varepsilon,n}(Z^{Y,\varepsilon,\circ})=O_{\mathbb{P}}(\varepsilon^{4})+O_{\mathbb{P}}(\varepsilon^{2})+O_{\mathbb{P}}\big(\varepsilon^{2}\sqrt{\varepsilon^{2}+\Delta_{n}}\big)=o_{\mathbb{P}}(\varepsilon).
\end{equation}
Finally, Theorems~\ref{thm:auxiliary-oracle-lse} and~\ref{thm:transfer-convergence} imply $\widehat{\beta}_{\varepsilon,n}^{Y,\circ}=O_{\mathbb{P}}(1)$. Together with \textbf{Step~1} and \eqref{numerator-convergence}, this gives
\[
	\widehat{\beta}_{\varepsilon,n}^{X,\circ}-\widehat{\beta}_{\varepsilon,n}^{Y,\circ}=\frac{\mathfrak{N}_{\varepsilon,n}(Z^{X,\varepsilon,\circ})-\mathfrak{N}_{\varepsilon,n}(Z^{Y,\varepsilon,\circ})}{\mathfrak{D}_{n}(Z^{X,\varepsilon,\circ})}-\widehat{\beta}_{\varepsilon,n}^{Y,\circ}\frac{\mathfrak{D}_{n}(Z^{X,\varepsilon,\circ})-\mathfrak{D}_{n}(Z^{Y,\varepsilon,\circ})}{\mathfrak{D}_{n}(Z^{X,\varepsilon,\circ})}=o_{\mathbb{P}}(\varepsilon).\qedhere
\]
\end{proof}
Combining~\eqref{auxiliary-oracle-under-original-measure} with Theorem~\ref{thm:auxiliary-cev-oracle-equivalence} yields consistency and the Gaussian limit for the CEV-based oracle:
\[
    \widehat{\beta}_{\varepsilon,n}^{X,\circ}\xrightarrow{\mathbb{P}\text{-prob.}}\beta_{0},\qquad\varepsilon^{-1}\big(\widehat{\beta}_{\varepsilon,n}^{X,\circ}-\beta_{0}\big)\xrightarrow{\mathbb{P}\text{-prob.}}\xi_{\beta_{0}},
\]
where $\xi_{\beta_{0}}$ and $V_{\beta_{0}}$ are as defined in Theorem~\ref{thm:auxiliary-oracle-lse}, with $\mathcal{L}_{\mathbb{P}}(\xi_{\beta_{0}})=\mathcal{N}(0,V_{\beta_{0}})$.

\subsection{Asymptotic rate of the preliminary estimator}\label{subsec:B3}
To remove the oracle dependence of the state mapping on the unknown elasticity $\vartheta_{0}$ so that the state mapping becomes feasible, we replace the unknown elasticity with the preliminary estimator $\widetilde{\vartheta}_{\varepsilon,n}$ defined in \eqref{preliminary-estimator}. It is constructed from the local realized variation of the observed small-noise CEV process, following the quadratic-variation principle underlying realized volatility; see \cite{barndorff2002econometric}. Related work by \cite{mishura2022parameter} applies this idea to CKLS elasticity estimation through quadratic-variation-based quantities, without explicitly treating realized-variation approximation under high-frequency infill sampling. Here, we adapt this idea to a high-frequency infill setting using realized variation over a shrinking window and establish the convergence rate required for the subsequent plug-in analysis.

\begin{lemma}[Corrected local-window approximation for the target CEV process]\label{lem:CEV-holder-continuity}
Let $X^{\varepsilon,\vartheta_{0}}$ be the target CEV process with fixed $\vartheta_{0}\in(1/2,1)$. Let $\Phi_{n}(\vartheta)=(e^{2\mu\vartheta h_{n}}-1)/(2\mu\vartheta h_{n})$. Under Assumption~\ref{ass:A4}, as $\varepsilon\to0$ and $n\to\infty$,
\begin{equation}\label{local-window-approximation-rate}
    \frac{1}{h_{n}}\int_{\tau_{n}}^{\tau_{n}'}\big(X_{s}^{\varepsilon,\vartheta_{0}}\big)^{2\vartheta_{0}}ds-\big(X_{\tau_{n}}^{\varepsilon,\vartheta_{0}}\big)^{2\vartheta_{0}}\Phi_{n}(\vartheta_{0})=O_{\mathbb{P}}\big(\varepsilon\sqrt{h_{n}}\big).
\end{equation}
\end{lemma}
\begin{proof}
\textbf{Step~1 (Uniform positive moments)}: Following the stopping argument in the proof of Lemma~\ref{lem:auxiliary-pseudo-cir-bounds}, define $\tau_{m}:=\inf\{t\geq0:X_{t}^{\varepsilon,\vartheta_{0}}\geq m\}$ for $m>x_{0}$. For $p\geq2$, H\"{o}lder's and BDG inequalities applied to the stopped SDE, together with $x^{p\vartheta_{0}}\leq1+x^{p}$, give
\[
    \mathbb{E}^{\mathbb{P}}\Big[\sup_{0\leq r\leq t}\big(X_{r\wedge\tau_{m}}^{\varepsilon,\vartheta_{0}}\big)^{p}
    \Big]\leq C_{p}\Big(1+\int_{0}^{t}\mathbb{E}^{\mathbb{P}}\Big[\sup_{0\leq r\leq s}\big(X_{r\wedge\tau_{m}}^{\varepsilon,\vartheta_{0}}\big)^{p}\Big]ds\Big),
\]
where $C_{p}$ is independent of $m$ and $0<\varepsilon\leq\varepsilon_{0}$. Gr\"{o}nwall's inequality, nonexplosion, and Fatou's lemma therefore yield, also using Lyapunov's inequality for $0<q<2$,
\[
    \sup_{0<\varepsilon\leq\varepsilon_{0}}\mathbb{E}^{\mathbb{P}}\Big[\sup_{0\leq t\leq T}\big(X_{t}^{\varepsilon,\vartheta_{0}}\big)^{q}\Big]<\infty,\qquad q>0.
\]
\textbf{Step~2 (Corrected window approximation)}: For $0\leq u\leq h_{n}$, applying the variation-of-constants formula to \eqref{small-noise-CEV-def} gives $X_{\tau_{n}+u}^{\varepsilon,\vartheta_{0}}=e^{\mu u}X_{\tau_{n}}^{\varepsilon,\vartheta_{0}}+\mathfrak R_{n,u}^{\varepsilon}$, where
\[
    \mathfrak{R}_{n,u}^{\varepsilon}:=\varepsilon\nu\int_{\tau_{n}}^{\tau_{n}+u}e^{\mu(\tau_{n}+u-s)}\big(X_{s}^{\varepsilon,\vartheta_{0}}\big)^{\vartheta_{0}}dW_{s}.
\]
It\^{o}'s isometry and \textbf{Step~1} imply $\mathbb{E}^{\mathbb{P}}[|\mathfrak{R}_{n,u}^{\varepsilon}|^{2}]\leq C\varepsilon^{2}u$. Using $|x^{2\vartheta_{0}}-y^{2\vartheta_{0}}|\leq C(x^{2\vartheta_{0}-1}+y^{2\vartheta_{0}-1})|x-y|$ for $x,y\geq0$, the Cauchy--Schwarz inequality and \textbf{Step~1} yield
\[
    \mathbb{E}^{\mathbb{P}}\Big[\Big|\big(X_{\tau_{n}+u}^{\varepsilon,\vartheta_{0}}\big)^{2\vartheta_{0}}-\big(e^{\mu u}X_{\tau_{n}}^{\varepsilon,\vartheta_{0}}\big)^{2\vartheta_{0}}\Big|\Big]\leq C\big(\mathbb{E}^{\mathbb{P}}\big[|\mathfrak{R}_{n,u}^{\varepsilon}|^{2}\big]\big)^{1/2}\leq C\varepsilon\sqrt{u}.
\]
Since $\Phi_{n}(\vartheta_{0})=h_{n}^{-1}\int_{0}^{h_{n}}e^{2\mu\vartheta_{0}u}du$, integration gives
\[
    \mathbb{E}^{\mathbb{P}}\Big[\Big|\frac{1}{h_{n}}\int_{\tau_{n}}^{\tau_{n}'}\big(X_{s}^{\varepsilon,\vartheta_{0}}\big)^{2\vartheta_{0}}ds-\big(X_{\tau_{n}}^{\varepsilon,\vartheta_{0}}\big)^{2\vartheta_{0}}\Phi_{n}(\vartheta_{0})\Big|\Big]\leq\frac{C\varepsilon}{h_{n}}\int_{0}^{h_{n}}\sqrt{u}du\leq C\varepsilon\sqrt{h_{n}}.
\]
Markov's inequality proves the assertion.
\end{proof}

\begin{lemma}[Approximation of local quadratic variation by realized variation]\label{lem:local-RV-QV-approximation}
Recall that
\[
	\widehat{\mathcal{V}}_{\varepsilon,n}:=\sum_{i=j_{n}+1}^{j_{n}+m_{n}}\Big(X_{t_{i}}^{\varepsilon,\vartheta_{0}}-X_{t_{i-1}}^{\varepsilon,\vartheta_{0}}\Big)^{2},\quad\text{where}\quad j_{n}:=\lfloor\tau/\Delta_{n}\rfloor, \quad\tau_{n}=t_{j_{n}},\quad h_{n}=m_{n}\Delta_{n}, \quad\tau'_{n}:=\tau_{n}+h_{n}\leq T.
\]
Under Assumptions~\ref{ass:A1},~\ref{ass:A2},~\ref{ass:A4}
\begin{equation}\label{local-RV-QV-full-rate}
    \frac{\widehat{\mathcal{V}}_{\varepsilon,n}-\big\{[X^{\varepsilon,\vartheta_{0}}]_{\tau'_{n}}-[X^{\varepsilon,\vartheta_{0}}]_{\tau_{n}}\big\}}{\varepsilon^{2}\nu^{2}h_{n}}=O_{\mathbb{P}}\Big(\sqrt{\frac{\Delta_{n}}{h_{n}}}+\frac{\Delta_{n}}{\varepsilon\sqrt{h_{n}}}+\frac{\Delta_{n}}{\varepsilon^{2}}\Big).
\end{equation}
\end{lemma}
\begin{proof}
Set
\[
    M_{t}:=\varepsilon\nu\int_{0}^{t}\big(X_{s}^{\varepsilon,\vartheta_{0}}\big)^{\vartheta_{0}}dW_{s},\qquad A_{i}:=\int_{t_{i-1}}^{t_{i}}X_{s}^{\varepsilon,\vartheta_{0}}ds.
\]
Throughout, sums run over $i=j_{n}+1,\ldots,j_{n}+m_{n}$, and $C$ is independent of $\varepsilon$ and $n$. We use the uniform positive-moment bounds established in the proof of Lemma~\ref{lem:CEV-holder-continuity}. Since $\Delta_{i}X^{\varepsilon,\vartheta_{0}}=\mu A_{i}+\Delta_{i}M$ and $[X^{\varepsilon,\vartheta_{0}}]=\langle M\rangle$, expanding the square gives
\begin{equation}\label{RV-QV-decomposition}
    \widehat{\mathcal{V}}_{\varepsilon,n}-\big\{[X^{\varepsilon,\vartheta_{0}}]_{\tau_{n}'}-[X^{\varepsilon,\vartheta_{0}}]_{\tau_{n}}\big\}=\mu^{2}\sum_{i}A_{i}^{2}+2\mu\sum_{i}A_{i}\Delta_{i}M+\sum_{i}\big[(\Delta_{i}M)^{2}-\Delta_{i}\langle M\rangle\big].
\end{equation}
\textbf{Step~1 (Drift-square term)}: By the Cauchy--Schwarz inequality and the moment bounds,
\begin{equation}\label{RV-QV-drift-square-rate}
    \mu^{2}\sum_{i}A_{i}^{2}\leq\mu^{2}\Delta_{n}\int_{\tau_{n}}^{\tau_{n}'}\big(X_{s}^{\varepsilon,\vartheta_{0}}\big)^{2}ds=O_{\mathbb{P}}(h_{n}\Delta_{n}).
\end{equation}
\textbf{Step~2 (Cross term)}: Write $A_{i}=X_{t_{i-1}}^{\varepsilon,\vartheta_{0}}\Delta_{n}+r_{i}$, where $r_{i}:=\int_{t_{i-1}}^{t_{i}}(X_{s}^{\varepsilon,\vartheta_{0}}-X_{t_{i-1}}^{\varepsilon,\vartheta_{0}})ds$. It\^{o}'s isometry and the moment bounds give
\begin{equation}\label{cross-leading-rate}
    \mathbb{E}^{\mathbb{P}}\Big[\Big|\Delta_{n}\sum_{i}X_{t_{i-1}}^{\varepsilon,\vartheta_{0}}\Delta_{i}M\Big|^{2}\Big]=\varepsilon^{2}\nu^{2}\Delta_{n}^{2}\sum_{i}\int_{t_{i-1}}^{t_{i}}\mathbb{E}^{\mathbb{P}}\Big[\big(X_{t_{i-1}}^{\varepsilon,\vartheta_{0}}\big)^{2}\big(X_{s}^{\varepsilon,\vartheta_{0}}\big)^{2\vartheta_{0}}\Big]ds\leq C\varepsilon^{2}h_{n}\Delta_{n}^{2}.
\end{equation}
For the remainder, the SDE, the Cauchy--Schwarz inequality, and It\^{o}'s isometry imply
\[
    \mathbb{E}^{\mathbb{P}}\Big[\big|X_{t}^{\varepsilon,\vartheta_{0}}-X_{s}^{\varepsilon,\vartheta_{0}}\big|^{2}\Big]\leq C\big((t-s)^{2}+\varepsilon^{2}(t-s)\big).
\]
Consequently,
\[
\begin{aligned}
    &\mathbb{E}^{\mathbb{P}}\Big[\sum_{i}r_{i}^{2}\Big]\leq\Delta_{n}\sum_{i}\int_{t_{i-1}}^{t_{i}}\mathbb{E}^{\mathbb{P}}\big|X_{s}^{\varepsilon,\vartheta_{0}}-X_{t_{i-1}}^{\varepsilon,\vartheta_{0}}\big|^{2}ds\leq C\big(h_{n}\Delta_{n}^{3}+\varepsilon^{2}h_{n}\Delta_{n}^{2}\big),\\
   & \mathbb{E}^{\mathbb{P}}\Big[\sum_{i}(\Delta_{i}M)^{2}\Big]=\varepsilon^{2}\nu^{2}\int_{\tau_{n}}^{\tau_{n}'}\mathbb{E}^{\mathbb{P}}\Big[\big(X_{s}^{\varepsilon,\vartheta_{0}}\big)^{2\vartheta_{0}}\Big]ds\leq C\varepsilon^{2}h_{n}.
\end{aligned}
\]
The Cauchy--Schwarz inequality therefore yields
\[
    \sum_{i}r_{i}\Delta_{i}M=O_{\mathbb{P}}\big(\varepsilon h_{n}\Delta_{n}^{3/2}+\varepsilon^{2}h_{n}\Delta_{n}\big)=o_{\mathbb{P}}\big(\varepsilon\Delta_{n}\sqrt{h_{n}}\big).
\]
Together with \eqref{cross-leading-rate}, this gives
\begin{equation}\label{RV-QV-cross-rate}
    \sum_{i}A_{i}\Delta_{i}M=O_{\mathbb{P}}\big(\varepsilon\Delta_{n}\sqrt{h_{n}}\big).
\end{equation}
\textbf{Step~3 (Martingale-square term)}: It\^{o}'s formula gives
\[
    (\Delta_{i}M)^{2}-\Delta_{i}\langle M\rangle=2\int_{t_{i-1}}^{t_{i}}(M_{s}-M_{t_{i-1}})dM_{s}.
\]
For $t_{i-1}\leq s\leq t_{i}$, the BDG inequality and the moment bounds yield $\mathbb{E}^{\mathbb{P}}[|M_{s}-M_{t_{i-1}}|^{4}]\leq C\varepsilon^{4}(s-t_{i-1})^{2}$. Hence, It\^{o}'s isometry and the Cauchy--Schwarz inequality imply
\[
\begin{aligned}
    \mathbb{E}^{\mathbb{P}}\Big[\Big|\sum_{i}\big[(\Delta_{i}M)^{2}-\Delta_{i}\langle M\rangle\big]\Big|^{2}]&=4\varepsilon^{2}\nu^{2}\sum_{i}\int_{t_{i-1}}^{t_{i}}\mathbb{E}^{\mathbb{P}}\Big[(M_{s}-M_{t_{i-1}})^{2}\big(X_{s}^{\varepsilon,\vartheta_{0}}\big)^{2\vartheta_{0}}\Big]ds\\
    &\leq C\varepsilon^{4}\sum_{i}\int_{t_{i-1}}^{t_{i}}(s-t_{i-1})ds\leq C\varepsilon^{4}h_{n}\Delta_{n}.
\end{aligned}
\]
Thus,
\begin{equation}\label{RV-QV-martingale-square-rate}
    \sum_{i}\big[(\Delta_{i}M)^{2}-\Delta_{i}\langle M\rangle\big]=O_{\mathbb{P}}\big(\varepsilon^{2}\sqrt{h_{n}\Delta_{n}}\big).
\end{equation}
Substituting \eqref{RV-QV-drift-square-rate}, \eqref{RV-QV-cross-rate}, and \eqref{RV-QV-martingale-square-rate} into \eqref{RV-QV-decomposition} and dividing by $\varepsilon^{2}\nu^{2}h_{n}$ proves \eqref{local-RV-QV-full-rate}.
\end{proof}

\begin{theorem}[Rate of the feasible preliminary estimator]
\label{thm:preliminary-estimator-rate}
Under Assumptions~\ref{ass:A1},~\ref{ass:A2},~\ref{ass:A4},
and~\ref{ass:A5},
\[
    \mathbb{P}(\mathcal{E}_{\varepsilon,n})\to1 \quad\text{and}\quad\widetilde{\vartheta}_{\varepsilon,n}-\vartheta_{0}=O_{\mathbb{P}}\big(\varepsilon \sqrt{h_{n}}\big)+O_{\mathbb{P}}\Big(\sqrt{\frac{\Delta_{n}}{h_{n}}}+\frac{\Delta_{n}}{\varepsilon\sqrt{h_{n}}}+\frac{\Delta_{n}}{\varepsilon^{2}}\Big).
\]
In particular, Assumption~\ref{ass:A5} reduces this rate to $\widetilde{\vartheta}_{\varepsilon,n}-\vartheta_{0}=o_{\mathbb{P}}(\varepsilon)$.
\end{theorem}
\begin{proof}
The normalized realized-variation error admits the decomposition
\[
    \frac{\widehat{\mathcal{V}}_{\varepsilon,n}}{\varepsilon^{2}\nu^{2}h_{n}}-\big(X_{\tau_{n}}^{\varepsilon,\vartheta_{0}}\big)^{2\vartheta_{0}}\Phi_{n}(\vartheta_{0})=\frac{\widehat{\mathcal{V}}_{\varepsilon,n}-\big\{[X^{\varepsilon,\vartheta_{0}}]_{\tau'_{n}}-[X^{\varepsilon,\vartheta_{0}}]_{\tau_{n}}\big\}}{\varepsilon^{2}\nu^{2}h_{n}
    }+\Big\{\frac{1}{h_{n}}\int_{\tau_{n}}^{\tau'_{n}}\big(X_{s}^{\varepsilon,\vartheta_{0}}\big)^{2\vartheta_{0}}ds-\big(X_{\tau_{n}}^{\varepsilon,\vartheta_{0}}\big)^{2\vartheta_{0}}\Phi_{n}(\vartheta_{0})\Big\}.
\]
Set $U_{\varepsilon,n}:=\widehat{\mathcal V}_{\varepsilon,n}/(\varepsilon^{2}\nu^{2}h_{n})$, $r_{\varepsilon,n}:=\varepsilon\sqrt{h_{n}}+\sqrt{\Delta_{n}/h_{n}}+\Delta_{n}/(\varepsilon\sqrt{h_{n}})+\Delta_{n}/\varepsilon^{2}$. Combining Lemmas~\ref{lem:CEV-holder-continuity},~\ref{lem:local-RV-QV-approximation} gives
\begin{equation}\label{normalized-rv-preliminary-expansion}
    U_{\varepsilon,n}-\big(X_{\tau_{n}}^{\varepsilon,\vartheta_{0}}\big)^{2\vartheta_{0}}\Phi_{n}(\vartheta_{0})=O_{\mathbb{P}}(r_{\varepsilon,n}).
\end{equation}
Assumption~\ref{ass:A4} controls approximation and local identification: $h_{n}\to0$ controls the stochastic within-window error, while $m_{n}\to\infty$ ensures increasingly many increments within the shrinking window. The condition $|\log\mathcal{X}_{\tau}|\geq c_{0}$ ensures nondegenerate local identification, since the limiting derivative of the corrected identification map is $2\log\mathcal{X}_{\tau}$, which vanishes if $\mathcal{X}_{\tau}=1$.

By Lemma~\ref{lem:uniform-deterministic-limit} and $\Phi_{n}(\vartheta_{0})\to1$, $\big(X_{\tau_{n}}^{\varepsilon,\vartheta_{0}}\big)^{2\vartheta_{0}}\Phi_{n}(\vartheta_{0})$ converges in probability to $\mathcal{X}_{\tau}^{2\vartheta_{0}}>0$. Since $r_{\varepsilon,n}\to0$ under Assumptions~\ref{ass:A4} and~\ref{ass:A5}, $U_{\varepsilon,n}$ has the same limit. Together with $|\log\mathcal{X}_{\tau}|\geq c_{0}$, this implies $\mathbb{P}(\mathcal{E}_{\varepsilon,n})\to1$.

On $\mathcal{E}_{\varepsilon,n}$, define the logarithmic residual
\[
    g_{\varepsilon,n}(\vartheta):=\log U_{\varepsilon,n}-2\vartheta\log X_{\tau_{n}}^{\varepsilon,\vartheta_{0}}-\log\Phi_{n}(\vartheta),
\]
and extend it by $0$ outside this event. The two quantities compared in \eqref{normalized-rv-preliminary-expansion} are bounded away from $0$ with probability tending to $1$. The mean value theorem for the logarithm therefore gives
\begin{equation}\label{corrected-log-preliminary-expansion}
    g_{\varepsilon,n}(\vartheta_{0})=O_{\mathbb{P}}(r_{\varepsilon,n}).
\end{equation}
By Lemma~\ref{lem:uniform-deterministic-limit} and $|\log\mathcal{X}_{\tau}|\geq c_{0}$, $|\log X_{\tau_{n}}^{\varepsilon,\vartheta_{0}}|\geq c_{0}/2$ with probability tending to $1$. Since $\sup_{\vartheta\in\Theta}|\partial_{\vartheta}\log\Phi_{n}(\vartheta)|=O(h_{n})$, for all sufficiently large $n$, on $\mathcal{E}_{\varepsilon,n}$,
\[
    \inf_{\vartheta\in\Theta}|\partial_{\vartheta}g_{\varepsilon,n}(\vartheta)|=\inf_{\vartheta\in\Theta}\big|2\log X_{\tau_n}^{\varepsilon,\vartheta_{0}}+\partial_\vartheta\log\Phi_n(\vartheta)\big|\geq c_{0}/2.
\]
Moreover, $\vartheta_{0}\in\Theta$ and the minimizing property of $\widetilde{\vartheta}_{\varepsilon,n}^{*}$ imply $|g_{\varepsilon,n}(\widetilde{\vartheta}_{\varepsilon,n}^{*})| \leq|g_{\varepsilon,n}(\vartheta_{0})|$.
Hence, on $\mathcal{E}_{\varepsilon,n}$, the mean value theorem and the triangle inequality yield
\[
    |\widetilde{\vartheta}_{\varepsilon,n}^{*}-\vartheta_{0}|\leq\frac{2}{c_{0}}\big|g_{\varepsilon,n}(\widetilde{\vartheta}_{\varepsilon,n}^{*})-g_{\varepsilon,n}(\vartheta_{0})\big|\leq\frac{4}{c_{0}}|g_{\varepsilon,n}(\vartheta_{0})|.
\]
Since $\widetilde{\vartheta}_{\varepsilon,n}=\widetilde{\vartheta}_{\varepsilon,n}^{*}$
on $\mathcal{E}_{\varepsilon,n}$ and $\mathbb{P}(\mathcal{E}_{\varepsilon,n})\to1$,
\eqref{corrected-log-preliminary-expansion} gives $\widetilde{\vartheta}_{\varepsilon,n}-\vartheta_{0}=O_{\mathbb{P}}(r_{\varepsilon,n})$.

Assumption~\ref{ass:A5} strengthens the basic joint small-noise and high-frequency regime. Together with $h_{n}\to0$ from Assumption~\ref{ass:A4}, it gives
\[
    \varepsilon\sqrt{h_{n}}=o(\varepsilon),\qquad\sqrt{\frac{\Delta_{n}}{h_{n}}}=o(\varepsilon),\qquad\frac{\Delta_{n}}{\varepsilon^2}=o(\varepsilon).
\]
The mixed term also satisfies
\[
    \frac{\Delta_{n}}{\varepsilon\sqrt{h_{n}}}=\sqrt{\frac{\Delta_{n}}{h_{n}}}\frac{\sqrt{\Delta_{n}}}{\varepsilon}=o(\varepsilon)o(1)=o(\varepsilon).
\]
Thus, $r_{\varepsilon,n}=o(\varepsilon)$, and the preceding rate bound yields
$\widetilde{\vartheta}_{\varepsilon,n}-\vartheta_{0}=o_{\mathbb{P}}(\varepsilon)$.
\end{proof}

\subsection{Parameter sensitivity and plug-in equivalence}\label{subsec:B4}
The CEV-based oracle limit established above uses the state mapping $\mathcal{T}$ evaluated at the unknown elasticity $\vartheta_{0}$. Implementing the estimator requires replacing $\vartheta_{0}$ by the preliminary estimator $\widetilde{\vartheta}_{\varepsilon,n}$, which changes the transformed observations and hence both the numerator and denominator of the least-squares estimator. This substitution requires separate justification at the first-order scale: preserving the oracle limit requires the resulting estimation error to be $o_{\mathbb{P}}(\varepsilon)$, rather than merely converging to $0$. We therefore examine the sensitivity of the estimator to the elasticity used in the state mapping. Recalling the path functional in \eqref{path-functional-B}, we express the oracle and feasible estimators as two evaluations of a common parameter-indexed map. For $\vartheta\in\Theta$, define
\[
	Z_{t_{i}}^{X,\varepsilon}(\vartheta):=\mathcal{T}\big(X_{t_{i}}^{\varepsilon,\vartheta_{0}};\vartheta\big)\text{ and }\mathcal{G}_{\varepsilon,n}(\vartheta):=\mathcal{B}_{\varepsilon,n}\big(Z^{X,\varepsilon}(\vartheta)\big)\quad\text{so that}\quad\mathcal{G}_{\varepsilon,n}(\vartheta_{0})=\widehat{\beta}_{\varepsilon,n}^{X,\circ}\text{ and }\mathcal{G}_{\varepsilon,n}(\widetilde{\vartheta}_{\varepsilon,n})=\widehat{\beta}_{\varepsilon,n}.
\]
By the mean value theorem,
\[
    \widehat{\beta}_{\varepsilon,n}-\widehat{\beta}_{\varepsilon,n}^{X,\circ}=\partial_{\vartheta}\mathcal{G}_{\varepsilon,n}(\bar{\vartheta}_{\varepsilon,n})\big(\widetilde{\vartheta}_{\varepsilon,n}-\vartheta_{0}\big)\qquad\text{for some $\bar{\vartheta}_{\varepsilon,n}$ between
$\widetilde{\vartheta}_{\varepsilon,n}$ and $\vartheta_{0}$.}
\]
Theorem~\ref{thm:preliminary-estimator-rate} gives $\widetilde{\vartheta}_{\varepsilon,n}-\vartheta_{0}=o_{\mathbb{P}}(\varepsilon)$, but transferring this rate to the plug-in error requires control of the derivative of $\mathcal{G}_{\varepsilon,n}$, which depends on the observations and varies with both $\varepsilon$ and $n$. A uniform $O_{\mathbb{P}}(1)$ bound ensures that the estimator does not amplify the preliminary estimation error beyond its order. The mean-value representation then yields an $o_{\mathbb{P}}(\varepsilon)$ difference between the feasible and CEV-based oracle estimators, preserving the oracle first-order limit. Because this comparison is pathwise, it does not require independence between the preliminary estimator and the observations used in the second step. The following theorem establishes the required sensitivity bound and the resulting plug-in equivalence.

\begin{theorem}[Parameter sensitivity and plug-in equivalence]
\label{thm:plugin-equivalence}
Suppose that Assumptions~\ref{ass:A1},~\ref{ass:A2},~\ref{ass:A3} hold. Then
\[
    \sup_{\vartheta\in\Theta}\big|\partial_{\vartheta}\mathcal{G}_{\varepsilon,n}(\vartheta)\big|=O_{\mathbb{P}}(1),
\]
If, in addition, Assumptions~\ref{ass:A4} and~\ref{ass:A5} hold, then
\[
    \widehat{\beta}_{\varepsilon,n}-\widehat{\beta}_{\varepsilon,n}^{X,\circ}=o_{\mathbb{P}}(\varepsilon).
\]
\end{theorem}
\begin{proof}
\textbf{Step~1 (Localization and uniform bounds)}: By Lemma~\ref{lem:uniform-deterministic-limit}, there exists a compact interval $K\Subset(0,\infty)$ such that $X^{\varepsilon,\vartheta_{0}}$ remains in $K$ on $[0,T]$ with probability tending to $1$. Denote this localization event by $\Omega_{K}^{\varepsilon}:=\{X_t^{\varepsilon,\vartheta_0}\in K \text{ for all }t\in[0,T]\}$, so that $\mathbb{P}(\Omega_{K}^{\varepsilon})\to1$. On $K\times\Theta$, the state mapping and its required derivatives are continuous, so
\begin{equation}\label{transform-uniform-derivative-bound}
    \sup_{(x,\vartheta)\in K\times\Theta}\big[|\mathcal{T}(x;\vartheta)|+|\partial_x\mathcal{T}(x;\vartheta)|+|\partial_{\vartheta}\mathcal{T}(x;\vartheta)|+|\partial_{x\vartheta}\mathcal{T}(x;\vartheta)|\big]<\infty.
\end{equation}
For $Z_{t_{i}}^{X,\varepsilon}(\vartheta)=\mathcal{T}(X_{t_{i}}^{\varepsilon,\vartheta_{0}};\vartheta)$, write $\dot{Z}_{t_{i}}^{X,\varepsilon}(\vartheta):=\partial_{\vartheta}\mathcal{T}(X_{t_{i}}^{\varepsilon,\vartheta_{0}};\vartheta)$. The preceding bound gives
\begin{equation}\label{uniform-Z-Zdot-bound}
    \sup_{\vartheta\in\Theta}\max_{0\leq i\leq n}\big(|Z_{t_{i}}^{X,\varepsilon}(\vartheta)|+|\dot{Z}_{t_{i}}^{X,\varepsilon}(\vartheta)|\big)=O_{\mathbb{P}}(1).
\end{equation}
Moreover, the SDE, the Cauchy--Schwarz inequality, It\^{o}'s isometry, and
\eqref{X-Y-uniform-second-moment} yield
\[
    \mathbb{E}^{\mathbb{P}}\Big[\sum_{i=1}^{n}|\Delta_{i}X^{\varepsilon,\vartheta_{0}}|^{2}\Big]\leq C(\varepsilon^{2}+\Delta_{n}).
\]
On $\Omega_{K}^{\varepsilon}$, the mean value theorem and \eqref{transform-uniform-derivative-bound} give
\[
    |\Delta_{i}Z^{X,\varepsilon}(\vartheta)|+|\Delta_{i}\dot{Z}^{X,\varepsilon}(\vartheta)|\leq C|\Delta_{i}X^{\varepsilon,\vartheta_0}|,\qquad\vartheta\in\Theta,
\]
with $C$ independent of $\vartheta$, $\varepsilon$, and $n$. Squaring and summing, the preceding moment estimate, Markov's inequality, and $\mathbb{P}(\Omega_{K}^{\varepsilon})\to1$ imply
\begin{equation}\label{uniform-Z-increment-bound}
    \sup_{\vartheta\in\Theta}\sum_{i=1}^{n}\big(|\Delta_{i}Z^{X,\varepsilon}(\vartheta)|^{2}+|\Delta_{i}\dot{Z}^{X,\varepsilon}(\vartheta)|^{2}\big)=O_{\mathbb{P}}(\varepsilon^{2}+\Delta_{n}).
\end{equation}
\textbf{Step~2 (Denominator and numerator)}: Define
\[
    \mathscr{D}_{\varepsilon,n}(\vartheta):=\Delta_{n}\sum_{i=1}^{n}\big(Z_{t_{i-1}}^{X,\varepsilon}(\vartheta)\big)^{2},\qquad\partial_{\vartheta}\mathscr{D}_{\varepsilon,n}(\vartheta)=2\Delta_{n}\sum_{i=1}^{n}Z_{t_{i-1}}^{X,\varepsilon}(\vartheta)\dot{Z}_{t_{i-1}}^{X,\varepsilon}(\vartheta).
\]
Since $n\Delta_{n}=T$, \eqref{uniform-Z-Zdot-bound} implies $\sup_{\vartheta\in\Theta}|\partial_{\vartheta}\mathscr{D}_{\varepsilon,n}(\vartheta)|=O_{\mathbb{P}}(1)$. By continuity and strict positivity on $K\times\Theta$, $c_{d}:=T\min_{(x,\vartheta)\in K\times\Theta}\mathcal{T}(x;\vartheta)^{2}>0$. On $\Omega_{K}^{\varepsilon}$, $(Z_{t_{i-1}}^{X,\varepsilon}(\vartheta))^{2}\geq c_{d}/T$ for every $i$ and $\vartheta\in\Theta$. Since $n\Delta_{n}=T$,
\[
    \inf_{\vartheta\in\Theta}\mathscr{D}_{\varepsilon,n}(\vartheta)\geq\Delta_{n}n\cdot c_{d}/T=c_{d}.
\]
Consequently, $\mathbb{P}(\inf_{\vartheta\in\Theta}\mathscr{D}_{\varepsilon,n}(\vartheta)\geq c_{d})\to1$. By \eqref{discrete-integration-by-parts-square}, the numerator of $\mathcal{G}_{\varepsilon,n}$ and its derivative can be written as 
\[
\begin{aligned}
    \mathscr{N}_{\varepsilon,n}(\vartheta)&=\alpha_{\varepsilon}\Delta_{n}\sum_{i=1}^{n}Z_{t_{i-1}}^{X,\varepsilon}(\vartheta)-\frac{1}{2}\Big[\big(Z_{T}^{X,\varepsilon}(\vartheta)\big)^{2}-\big(Z_{0}^{X,\varepsilon}(\vartheta)\big)^{2}-\sum_{i=1}^{n}\big(\Delta_{i}Z^{X,\varepsilon}(\vartheta)\big)^{2}\Big],\\
    \partial_{\vartheta}\mathscr{N}_{\varepsilon,n}(\vartheta)&=\alpha_{\varepsilon}\Delta_{n}\sum_{i=1}^{n}\dot{Z}_{t_{i-1}}^{X,\varepsilon}(\vartheta)-Z_{T}^{X,\varepsilon}(\vartheta)\dot{Z}_{T}^{X,\varepsilon}(\vartheta)+Z_{0}^{X,\varepsilon}(\vartheta)\dot{Z}_{0}^{X,\varepsilon}(\vartheta)+\sum_{i=1}^{n}\Delta_{i}Z^{X,\varepsilon}(\vartheta)\Delta_{i}\dot{Z}^{X,\varepsilon}(\vartheta).
\end{aligned}
\]
Using $2|xy|\leq x^{2}+y^{2}$ and \eqref{uniform-Z-increment-bound}, we obtain
\[
    \sup_{\vartheta\in\Theta}\Big|\sum_{i=1}^{n}\Delta_{i}Z^{X,\varepsilon}(\vartheta)\Delta_{i}\dot{Z}^{X,\varepsilon}(\vartheta)\Big|\leq\frac{1}{2}\sup_{\vartheta\in\Theta}\sum_{i=1}^{n}\big(|\Delta_{i}Z^{X,\varepsilon}(\vartheta)|^{2}+|\Delta_{i}\dot{Z}^{X,\varepsilon}(\vartheta)|^{2}\big)=O_{\mathbb{P}}(\varepsilon^{2}+\Delta_{n}).
\]
By \eqref{uniform-Z-Zdot-bound}, $n\Delta_{n}=T$, and $\alpha_{\varepsilon}=O(\varepsilon^{2})$, the first terms in $\mathscr{N}_{\varepsilon,n}(\vartheta)$ and $\partial_{\vartheta}\mathscr{N}_{\varepsilon,n}(\vartheta)$ are uniformly $O_{\mathbb P}(\varepsilon^2)$, while the endpoint terms are uniformly $O_{\mathbb P}(1)$. The squared-increment sum is controlled by \eqref{uniform-Z-increment-bound}. Thus,
\[
    \sup_{\vartheta\in\Theta}|\partial_\vartheta\mathscr{N}_{\varepsilon,n}(\vartheta)|=O_{\mathbb{P}}(1),\qquad\sup_{\vartheta\in\Theta}|\mathscr{N}_{\varepsilon,n}(\vartheta)|=O_{\mathbb{P}}(1).
\]
\textbf{Step~3 (Parameter sensitivity and plug-in equivalence)}: Since $\mathcal{G}_{\varepsilon,n}=\mathscr{N}_{\varepsilon,n}/\mathscr{D}_{\varepsilon,n}$, the quotient rule gives
\[
    \partial_{\vartheta}\mathcal{G}_{\varepsilon,n}(\vartheta)=\frac{\partial_{\vartheta}\mathscr{N}_{\varepsilon,n}(\vartheta)}{\mathscr{D}_{\varepsilon,n}(\vartheta)}-\frac{\mathscr{N}_{\varepsilon,n}(\vartheta)\partial_{\vartheta}\mathscr{D}_{\varepsilon,n}(\vartheta)}{\mathscr{D}_{\varepsilon,n}(\vartheta)^{2}}.
\]
The bounds in \textbf{Step~2} imply $\sup_{\vartheta\in\Theta}|\partial_{\vartheta}\mathcal{G}_{\varepsilon,n}(\vartheta)|=O_{\mathbb{P}}(1)$.

Under Assumptions~\ref{ass:A4} and~\ref{ass:A5}, Theorem~\ref{thm:preliminary-estimator-rate} gives $\widetilde\vartheta_{\varepsilon,n}-\vartheta_{0}=o_{\mathbb{P}}(\varepsilon)$. Since $\Theta$ is an interval containing both $\widetilde\vartheta_{\varepsilon,n}$ and $\vartheta_{0}$, the mean value theorem yields
\[
    \big|\widehat\beta_{\varepsilon,n}-\widehat\beta_{\varepsilon,n}^{X,\circ}\big|\leq\sup_{\vartheta\in\Theta}|\partial_{\vartheta}\mathcal{G}_{\varepsilon,n}(\vartheta)||\widetilde\vartheta_{\varepsilon,n}-\vartheta_{0}|=o_{\mathbb{P}}(\varepsilon).\qedhere
\]
\end{proof}
\begin{remark}[Differentiability under localization]
The parameter-sensitivity analysis is local in the state variable. Although $\mathcal{T}(x;\vartheta)$ is smooth in $\vartheta$ for $x>0$, the increment bounds established above require the mixed derivative $\partial_{x\vartheta}\mathcal{T}(x;\vartheta)$, which is singular as $x\downarrow0$. We therefore work on localization sets bounded away from $0$. On the other hand, the deterministic limit
$\mathcal{X}_{t}=x_{0}e^{\mu t}$ satisfies
$\inf_{0\leq t\leq T}\mathcal{X}_{t}>0$. By
Lemma~\ref{lem:uniform-deterministic-limit}, there exists a compact
interval $K\Subset(0,\infty)$ such that $\mathbb{P}\big(X_{t}^{\varepsilon,\vartheta_{0}}\in K \text{ for all }t\in[0,T]\big)\to1$. Since $\Theta\Subset(1/2,1)$, all derivatives of $\mathcal{T}$ required in the sensitivity analysis are uniformly bounded on $K\times\Theta$. Hence, on these high-probability localization sets, $\mathcal{G}_{\varepsilon,n}(\vartheta)$ is differentiable in $\vartheta$, and the mean-value argument comparing $\mathcal{G}_{\varepsilon,n}(\widetilde{\vartheta}_{\varepsilon,n})$ and $\mathcal{G}_{\varepsilon,n}(\vartheta_{0})$ is justified.
\end{remark}

\end{appendix}

\begingroup
\small
\setlength{\baselineskip}{\baselineskip}
\bibliographystyle{abbrv}
\bibliography{references}

@article{ait2002maximum,
  title={Maximum likelihood estimation of discretely sampled diffusions: a closed-form approximation approach},
  author={A{\"{i}}t-Sahalia, Yacine},
  journal={Econometrica},
  volume={70},
  number={1},
  pages={223--262},
  year={2002},
  publisher={Wiley Online Library}
}

@article{alaya2013asymptotic,
  title={Asymptotic behavior of the maximum likelihood estimator for ergodic and nonergodic square-root diffusions},
  author={Ben Alaya, Mohamed and Kebaier, Ahmed},
  journal={Stoch. Anal. Appl.},
  volume={31},
  number={4},
  pages={552--573},
  year={2013},
  publisher={Taylor \& Francis}
}

@article{andersen2007moment,
  title={Moment explosions in stochastic volatility models},
  author={Andersen, Leif B. G. and Piterbarg, Vladimir V.},
  journal={Finance Stoch.},
  volume={11},
  number={1},
  pages={29--50},
  year={2007},
  publisher={Springer}
}

@article{beckers1980constant,
  title={The constant elasticity of variance model and its implications for option pricing},
  author={Beckers, Stan},
  journal={J. Finance},
  volume={35},
  number={3},
  pages={661--673},
  year={1980},
  publisher={JSTOR}
}

@article{ballestra2011constant,
  title={The constant elasticity of variance model: calibration, test and evidence from the {I}talian equity market},
  author={Ballestra, Luca Vincenzo and Pacelli, Graziella},
  journal={Appl. Financ. Econ.},
  volume={21},
  number={20},
  pages={1479--1487},
  year={2011},
  publisher={Taylor \& Francis}
}

@article{ballestra2016numerical,
  title={A numerical method to estimate the parameters of the {CEV} model implied by {A}merican option prices: Evidence from {NYSE}},
  author={Ballestra, Luca Vincenzo and Cecere, Liliana},
  journal={Chaos, Solitons \& Fractals},
  volume={88},
  pages={100--106},
  year={2016},
  publisher={Elsevier}
}

@article{barndorff2002econometric,
  title={Econometric analysis of realized volatility and its use in estimating stochastic volatility models},
  author={Barndorff-Nielsen, Ole E and Shephard, Neil},
  journal={J. R. Stat. Soc. Ser. B},
  volume={64},
  number={2},
  pages={253--280},
  year={2002},
  publisher={Oxford University Press}
}

@article{barndorff2004power,
  title={Power and bipower variation with stochastic volatility and jumps},
  author={Barndorff-Nielsen, Ole E and Shephard, Neil},
  journal={J. Financ. Econom.},
  volume={2},
  number={1},
  pages={1--37},
  year={2004},
  publisher={Oxford University Press}
}

@article{beiranvand2021new,
  title={A new methodology to estimate constant elasticity of variance},
  author={Beiranvand, Ali and Ivaz, Karim and Beiranvand, Hamzeh},
  year={2021},
  journal={Comput. Methods Differ. Equ.},
  volume={9},
  nombuer={4},
  pages={1059--1068}
}

@article{bibby1995martingale,
  title={Martingale estimation functions for discretely observed diffusion processes},
  author={Bibby, Bo Martin and S{\o}rensen, Michael},
  journal={Bernoulli},
  volume={1(1--2)},
  pages={17--39},
  year={1995},
  publisher={JSTOR}
}

@article{chan1992empirical,
  title={An empirical comparison of alternative models of the short-term interest rate},
  author={Chan, K. Ceajer and Karolyi, George Andrew and Longstaff, Francis A. and Sanders, Anthony B.},
  journal={J. Finance},
  volume={47},
  number={3},
  pages={1209--1227},
  year={1992},
  publisher={Wiley Online Library}
}

@article{cox1985theory,
  title={A theory of the term structure of interest rates},
  author={Cox, John C. and Ingersoll, Jr., Jonathan E. and Ross, Stephen A.},
  journal={Econometrica},
  volume={53},
  number={2},
  pages={385--407},
  year={1985},
  publisher={World Scientific}
}

@article{delbaen2002note,
  title={A note on option pricing for the constant elasticity of variance model},
  author={Delbaen, Freddy and Shirakawa, Hiroshi},
  journal={Asia-Pac. Financ. Mark.},
  volume={9},
  number={2},
  pages={85--99},
  year={2002},
  publisher={Springer}
}

@article{emanuel1982further,
  title={Further results on the constant elasticity of variance call option pricing model},
  author={Emanuel, David C and MacBeth, James D},
  journal={J. Financ. Quant. Anal.},
  volume={17},
  number={4},
  pages={533--554},
  year={1982},
  publisher={Cambridge University Press}
}

@article{florens1989approximate,
  title={Approximate discrete-time schemes for statistics of diffusion processes},
  author={Florens-Zmirou, Danielle},
  journal={Statistics},
  volume={20},
  number={4},
  pages={547--557},
  year={1989},
  publisher={Taylor \& Francis}
}

@article{genon1990maximnm,
  title={Maximum contrast estimation for diffusion processes from discrete observations},
  author={Genon-Catalot, Valentine},
  journal={Statistics},
  volume={21},
  number={1},
  pages={99--116},
  year={1990},
  publisher={Taylor \& Francis}
}

@article{genon1993estimation,
  title={On the estimation of the diffusion coefficient for multi-dimensional diffusion processes},
  author={Genon-Catalot, Valentine and Jacod, Jean},
  journal={Ann. Inst. Henri Poincaré Probab. Stat.},
  volume={29},
  number={1},
  pages={119--151},
  year={1993}
}

@article{genon1994estimation,
  title={Estimation of the diffusion coefficient for diffusion processes: random sampling},
  author={Genon-Catalot, Valentine and Jacod, Jean},
  journal={Scand. J. Stat.},
  volume={21},
  number={3},
  pages={193--221},
  year={1994},
  publisher={JSTOR}
}

@article{gloter2009estimation,
  title={Estimation for stochastic differential equations with a small diffusion coefficient},
  author={Gloter, Arnaud and S{\o}rensen, Michael},
  journal={Stochastic Process. Appl.},
  volume={119},
  number={3},
  pages={679--699},
  year={2009},
  publisher={Elsevier}
}

@article{gontis2011interplay,
  title={Interplay between positive feedbacks in the generalized {CEV} process},
  author={Reimann, St. and Gontis, V. and Alaburda, M.},
  journal={Physica A},
  volume={390},
  number={8},
  pages={1393--1401},
  year={2011},
  publisher={Elsevier}
}

@article{guy2014parametric,
  title={Parametric inference for discretely observed multidimensional diffusions with small diffusion coefficient},
  author={Guy, Romain and Lar{\'e}do, Catherine and Vergu, Elisabeta},
  journal={Stochastic Processes and their Applications},
  volume={124},
  number={1},
  pages={51--80},
  year={2014},
  publisher={Elsevier}
}

@article{jacobsen2001discretely,
  title={Discretely Observed Diffusions: Classes of Estimating Functions and Small {$\Delta$}-optimality},
  author={Jacobsen, Martin},
  journal={Scand. J. Stat.},
  volume={28},
  number={1},
  pages={123--149},
  year={2001},
  publisher={Wiley Online Library}
}

@article{jacobsen2002optimality,
  title={Optimality and small {$\Delta$}-optimality of martingale estimating functions},
  author={Jacobsen, Martin},
  journal={Bernoulli},
  volume={8},
  number={5},
  pages={643--668},
  year={2002}
}

@article{jakobsen2017efficient,
  title={Efficient estimation for diffusions sampled at high frequency over a fixed time interval},
  author={Jakobsen, Nina Munkholt and S{\o}rensen, Michael},
  year={2017},
  journal={Bernoulli},
  volume={23},
  number={3},
  pages={1874--1910}
}

@article{kawai2023adaptive,
  title={Adaptive inference for small diffusion processes based on sampled data},
  author={Kawai, Tetsuya and Uchida, Masayuki},
  journal={Metrika},
  volume={86},
  number={6},
  pages={643--696},
  year={2023},
  publisher={Springer}
}

@article{kessler1999estimating,
  title={Estimating equations based on eigenfunctions for a discretely observed diffusion process},
  author={Kessler, Mathieu and S{\o}rensen, Michael},
  year={1999},
  journal={Bernoulli}, 
  volume={5},
  number={2},
  pages={299--314}
}

@article{laredo1990sufficient,
  title={A sufficient condition for asymptotic sufficiency of incomplete observations of a diffusion process},
  author={Laredo, Catherine F},
  journal={Ann. Stat.},
  volume={18},
  number={3},
  pages={1158--1171},
  year={1990},
  publisher={JSTOR}
}

@article{lindsay2012simulation,
  title={Simulation of the {CEV} process and the local martingale property},
  author={Lindsay, Alan E and Brecher, DR},
  journal={Math. Comput. Simulation},
  volume={82},
  number={5},
  pages={868--878},
  year={2012},
  publisher={Elsevier}
}

@article{long2013least,
  title={Least squares estimators for discretely observed stochastic processes driven by small {L}{\'e}vy noises},
  author={Long, Hongwei and Shimizu, Yasutaka and Sun, Wei},
  journal={J. Multivar. Anal.},
  volume={116},
  pages={422--439},
  year={2013},
  publisher={Elsevier}
}

@article{long2017least,
  title={Least squares estimators for stochastic differential equations driven by small {L}{\'e}vy noises},
  author={Long, Hongwei and Ma, Chunhua and Shimizu, Yasutaka},
  journal={Stochastic Processes and their Applications},
  volume={127},
  number={5},
  pages={1475--1495},
  year={2017},
  publisher={Elsevier}
}

@article{mijatovic2012martingale,
  title={On the martingale property of certain local martingales},
  author={Mijatovi{\'c}, Aleksandar and Urusov, Mikhail},
  journal={Probab. Theory Relat. Fields},
  volume={152},
  number={1},
  pages={1--30},
  year={2012},
  publisher={Springer}
}

@article{mishura2022parameter,
  title={Parameter estimation in {CKLS} model by continuous observations},
  author={Mishura, Yuliya and Ralchenko, Kostiantyn and Dehtiar, Olena},
  journal={Stat. Probab. Lett.},
  volume={184},
  pages={109391},
  year={2022},
  publisher={Elsevier}
}

@article{shimizu2017threshold,
  title={Threshold estimation for stochastic processes with small noise},
  author={Shimizu, Yasutaka},
  journal={Scand. J. Stat.},
  volume={44},
  number={4},
  pages={951--988},
  year={2017},
  publisher={Wiley Online Library}
}

@article{shoji1998estimation,
  title={Estimation for nonlinear stochastic differential equations by a local linearization method},
  author={Shoji, Isao and Ozaki, Tohru},
  journal={Stoch. Anal. Appl.},
  volume={16},
  number={4},
  pages={733--752},
  year={1998},
  publisher={Taylor \& Francis}
}

@article{sorensen2003small,
  title={Small-diffusion asymptotics for discretely sampled stochastic differential equations},
  author={S{\o}rensen, Michael and Uchida, Masayuki},
  journal={Bernoulli},
  volume={9},
  number={6},
  pages={1051--1069},
  year={2003},
  publisher={Bernoulli Society for Mathematical Statistics and Probability}
}

@article{tsumurai2020malliavin,
  title={{M}alliavin Differentiability of {CEV}-Type {H}eston Model},
  author={Tsumurai, Shota},
  journal={J. Math. Finance},
  volume={10},
  number={1},
  pages={173--199},
  year={2020},
  publisher={Scientific Research Publishing}
}

@article{tucker1988tests,
  title={Tests of the {B}lack-{S}choles and constant elasticity of variance currency call option valuation models},
  author={Tucker, Alan L and Peterson, David R and Scott, Elton},
  journal={J. Financ. Res.},
  volume={11},
  number={3},
  pages={201--214},
  year={1988},
  publisher={Wiley Online Library}
}

@article{uchida2013quasi,
  title={Quasi likelihood analysis of volatility and nondegeneracy of statistical random field},
  author={Uchida, Masayuki and Yoshida, Nakahiro},
  journal={Stochastic Process. Appl.},
  volume={123},
  number={7},
  pages={2851--2876},
  year={2013},
  publisher={Elsevier}
}

@article{xiao2017estimating,
  title={Estimating the constant elasticity of variance model with data-driven {M}arkov chain {M}onte {C}arlo methods},
  author={Xiao, Shuang and Li, Guo and Jia, Yunjing},
  journal={Asia-Pac. J. Oper. Res.},
  volume={34},
  number={01},
  pages={1740009},
  year={2017},
  publisher={World Scientific}
}

@article{yoshida1992estimation,
  title={Estimation for diffusion processes from discrete observation},
  author={Yoshida, Nakahiro},
  journal={J. Multivar. Anal.},
  volume={41},
  number={2},
  pages={220--242},
  year={1992},
  publisher={Elsevier}
}

@article{yuen2001estimation,
  title={Estimation in the constant elasticity of variance model},
  author={Yuen, K. C. and Yang, H and Chu, K. L.},
  journal={Br. Actuar. J.},
  volume={7},
  number={2},
  pages={275--292},
  year={2001},
  publisher={Cambridge University Press}
}

@techreport{sorensen2000small,
  author={S{\o}rensen, Michael},
  title={Small dispersion asymptotics for diffusion martingale estimating functions},
  institution={Department of Statistics and Operation Research, University of Copenhagen},
  number={Preprint No. 2000-2},
  address={Copenhagen},
  year={2000}
}

@unpublished{cox1975notes,
  author={Cox, John C.},
  title={Notes on Option Pricing {I}: Constant Elasticity of Variance Diffusions},
  note={Unpublished draft, Stanford University},
  year={1975}
}

@article{ning2025ckls,
  title={From {CKLS} Process to {CIR}-type and {OU}-type Processes: Using a Twice-differentiable Mapping and Generalized {Girsanov}'s Theorem},
  author={Ning, Boyuan and Shimizu, Yasutaka},
  journal={Asia-Pac. Financ. Mark.},
  year={2025},
  publisher={Springer},
  doi={10.1007/s10690-025-09563-1}
}

@book{revuz1999continuous,
  author={Revuz, Daniel and Yor, Marc},
  title={Continuous Martingales and Brownian Motion},
  edition={3rd},
  publisher={Springer},
  year={1999}
}

@book{jeanblanc2009mathematical,
  title={Mathematical methods for financial markets},
  author={Jeanblanc, Monique and Yor, Marc and Chesney, Marc},
  year={2009},
  publisher={Springer London}
}

@book{kutoyants1994identification,
  author={Kutoyants, Yury A.},
  title={Identification of Dynamical Systems with Small Noise},
  year={1994},
  publisher={Kluwer Academic Publishers},
  address={Dordrecht}
}
\endgroup

\end{document}